\documentclass[10pt,reqno]{article}

\usepackage{amsmath}
\usepackage{amssymb}
\usepackage{mathtools}
\usepackage{amsthm}
\usepackage{mathrsfs}
\usepackage{amsfonts}
\usepackage{amscd}

\usepackage{graphicx}
\usepackage[export]{adjustbox}
\usepackage{xcolor}
\usepackage[dvipsnames]{xcolor}
\usepackage{colortbl}

\usepackage[margin=1in]{geometry}
\usepackage{setspace}
\usepackage{float}
\usepackage{authblk}
\usepackage{enumitem}
\usepackage{subcaption}

\usepackage{hyperref}
\usepackage[nocompress, space]{cite}

\usepackage{esint}
\usepackage{lineno}
\usepackage{comment}

\theoremstyle{plain}
\newtheorem{thm}{Theorem}[section]
\newtheorem*{thm*}{Theorem}

\newtheorem*{cor*}{Corollary}
\newtheorem{lemma}[thm]{Lemma}
\newtheorem*{lemma*}{Lemma}
\newtheorem{prop}[thm]{Proposition}
\newtheorem*{prop*}{Proposition}

\theoremstyle{definition}
\newtheorem{egsample}[thm]{Example}
\newtheorem*{egsample*}{Example}
\newtheorem{definition} [thm]{Definition}

\newtheorem{problem}{Problem}
\counterwithout{problem}{section}

\newtheorem{remark}[thm]{Remark}
\newtheorem*{remark*}{Remark}

\newenvironment{eg*}
  {\begin{egsample*}}
  {\leavevmode\unskip\penalty9999 \hbox{}\nobreak\hfill\quad\hbox{\(\blacktriangle\)}\end{egsample*}}

\theoremstyle{plain}
\newtheorem{theorem}[thm]{Theorem}
\newtheorem{proposition}[thm]{Proposition}
\newtheorem{corollary}[thm]{Corollary}

\numberwithin{equation}{section}

\newtheorem{assumption}{Assumption}
\counterwithout{assumption}{section}

\theoremstyle{plain}

\newtheorem{innercustomAssumA}{Assumption}
\newenvironment{customAssumA}[1]{
    \renewcommand\theinnercustomAssumA{#1}
    \innercustomAssumA
}{\endinnercustomAssumA}

\DeclareMathOperator*{\Div}{div}

\DeclareMathOperator{\tr}{\mathfrak{tr}}

\newcommand{\R}{\mathbb{R}}

\def \bf{\textbf}

\newcommand{\cU}{\mathcal{U}}
\newcommand{\cX}{\mathcal{X}}
\newcommand{\cY}{\mathcal{Y}}

\def \cI {\mathcal{I}}

\newcommand{\abs}[1]{\left| #1\right|}

\newcommand{\del}{\partial}

\def \ep {\varepsilon}
\def \a {\alpha}
\def \b {\beta}

\newcommand{\widebar}[1]{\mkern1mu\overline{\mkern-1mu#1\mkern-1mu}\mkern1mu}

\def\d{{\,\, \rm d}}
\def\dx{{\,\,\rm d}x}
\def\dy{{\,\,\rm d}y}
\def\dz{{\,\,\rm d}z}

\def\ds{{\,\,\rm d}s}

\hypersetup{
    colorlinks=true,
    linkcolor=blue,
    citecolor=blue,
    urlcolor=blue,
}

\begin{document}

\title{{A First-Order Mean-Field Game on a Bounded Domain with Mixed Boundary Conditions}\thanks{The research reported in this publication was supported by funding from King Abdullah University of Science and Technology (KAUST), including baseline funds and KAUST OSR-CRG2021-4674. A. Alharbi is a Teaching Assistant at the Islamic University of Madinah, currently pursuing his PhD at KAUST.}	}

\author[1,2]{AbdulRahman M. Alharbi}
\author[1]{Yuri Ashrafyan}
\author[1]{Diogo Gomes}

\affil[1]{King Abdullah University of Science and Technology (KAUST)}
\affil[2]{Islamic University of Madinah}

\maketitle

\begin{abstract}
Entry-exit dynamics are crucial in modeling crowd movement. Here,
we present a novel first-order, stationary mean-field game (MFG) model on bounded domains that accurately captures entry-exit dynamics.
In our model, the interior dynamics are governed by a standard first-order stationary MFG system: a first-order Hamilton-Jacobi equation coupled with a transport equation.
The model incorporates mixed boundary conditions that correspond to an entry region $\Gamma_N$
and an exit region $\Gamma_D$. A Neumann condition on \(\Gamma_N\) prescribes the agent inflow via a non-homogeneous flux term, \(j(x)\); a no-entry condition on \(\Gamma_D\) restricts this boundary region to exit only, preventing inward flow; finally, in \(\Gamma_D\), we prescribe an upper bound on the exit cost combined with a complementary contact-set condition.

This contact-set condition identifies boundary points where the value function attains the exit cost (contact points) versus points where the non-penetration condition prevents artificial inflows (non-contact points). However, as our examples show, contact does not necessarily imply that exit occurs. This mixed approach overcomes the limitations of classical Dirichlet conditions, which
can artificially force boundary points to act as both entry and exit sites.

We analyze the system using a variational formulation, applying the direct method of calculus of variations to establish the existence of solutions under minimal regularity assumptions. Furthermore, we prove the uniqueness of the gradient of the value function (particularly in regions with positive agent density) and the uniqueness of the density function.

Several examples, including cases in one and two dimensions, illustrate first-order MFG phenomena such as the formation of empty regions (where agent density vanishes) and the proper assignment of entry and exit roles. These results establish a rigorous mathematical foundation for modeling realistic entry-exit scenarios.
\end{abstract}

\section{Introduction} 

This paper studies first-order stationary Mean-Field Games (MFGs) in bounded domains under nonstandard mixed boundary conditions (BCs).
These boundary conditions induce a novel MFG structure that resolves an inherent issue with standard Dirichlet conditions: a boundary point may inadvertently serve as an entrance instead of an exit.
This MFG models a scenario where agents can enter or exit only through designated boundary portions.
This new MFG system is relevant in various applications, such as crowd dynamics, where individuals may join or leave a venue through different gates, or economic models, where firms may enter or exit a market under certain conditions.
One of the paper's main goals is to show that the mixed BCs lead to a well-defined MFG formulation.

Let \(\mathcal{H}^{d-1}\) denote the \(d-1\)-dimensional Hausdorff measure. 
The precise problem we consider is as follows.

\begin{problem}[\bfseries The Mean-Field Game (MFG)]\label{problem:1}
Let \(\Omega \subset \mathbb{R}^d\) be a bounded, open, and connected \(C^1\)--domain whose boundary \(\Gamma = \partial \Omega\) 
{ 
can be written as \( \Gamma=\widebar{\Gamma}_D\cup\widebar{\Gamma}_N \), where 
\( \Gamma_D \) and \( \Gamma_N \) are  disjoint, smooth \((d-1)\)--dimensional  manifolds,  relatively open with respect to \(\Gamma\),
with
\begin{equation}
\label{inter}
\mathcal{H}^{d-1}(\widebar{\Gamma}_D \cap \widebar{\Gamma}_N) = 0,\quad 
\mathcal{H}^{d-1}(\Gamma_D) > 0,\quad \text{and}\quad \mathcal{H}^{d-1}(\Gamma_N) > 0.
\end{equation}
}
 Let
    \[
    j\colon \Gamma_N\to[0,\infty),\enskip\, \psi\colon \Gamma_D\to\mathbb{R},\enskip\, H\colon \Omega\times\mathbb{R}^d\to\mathbb{R},\enskip\, \text{and }\enskip g\colon[0,\infty)\to \R
    \]
    be given functions such that the \(C^1\)--mapping \( p\mapsto H(x,p) \) is convex and that \( g \) is continuous and increasing.
Solve the following system for \((m,u)\).
\begin{equation}
	\label{eq:MFG}
	\begin{cases}
		H(x, Du) = g(m)    & \text{ in } \Omega, \\
		- \Div(m D_pH(x, Du)) = 0 & \text{ in } \Omega,
	\end{cases}
\end{equation}
with the boundary conditions
\begin{equation}
	\label{eq:BC}
	\begin{cases}
		m D_pH(x, Du) \cdot \nu = j(x)                                     & \text{ on } \Gamma_N, \\
		m D_pH(x, Du)\cdot \nu \leq 0 \, \text{ and } \, u(x) \leq \psi(x) & \text{ on } \Gamma_D, \\
		(\psi-u) m D_pH(x, Du) \cdot \nu = 0                               & \text{ on } \Gamma_D.
	\end{cases}
\end{equation}
\end{problem}

The previous problem corresponds to the equilibrium state of a game in a bounded
domain with a boundary divided into two regions: an entry region \(\Gamma_N\) and an exit region
\( \Gamma_D \).
{ 
The first condition in \eqref{inter}
ensures the interface is negligible with respect to the boundary measure.
}

On \(\Gamma_N\), we prescribe the inflow of agents via a Neumann condition.
{ 
The non-homogeneous Neumann condition (i.e., \(j(x) \not\equiv 0\) on \(\Gamma_N\))
prescribes the influx of agents. While models with homogeneous Neumann conditions often
impose a total mass constraint (e.g., \(\int_\Omega m\,dx = 1\)) to prevent the trivial solution
\(m \equiv 0\), our formulation focuses on dynamics driven by specified boundary flows.
}

On \(\Gamma_D\), we enforce a relaxed Dirichlet condition on the value function;
the complementary condition on \(\Gamma_D\) ensures it functions solely as an exit region with $\psi$ as exit cost, avoiding artificial inflows.  We denote by
\[
\Gamma_D^1 := \{x\in\Gamma_D : u(x)=\psi(x)\}
\]
the \emph{contact set} (where the value function attains the exit cost) and by
\[
\Gamma_D^0 := \Gamma_D\setminus\Gamma_D^1
\]
the \emph{non-contact set}. In the non-contact set, the condition $(\psi-u) m D_pH(x, Du) \cdot \nu = 0 $ implies zero flow $m D_pH(x, Du) \cdot \nu = 0$, thus preventing artificial inflows through the exit that can arise if we impose a strict Dirichlet boundary condition.
Additional assumptions on functions \(H\), \(g\), \(j\) and \(\psi\) are given in Section~\ref{sec:Assumptions}.
In Section~\ref{subsec:example1111}, we present several examples that illustrate key phenomena predicted by our model, such as the formation of empty regions.

\subsection{The derivation of the model} \label{Intro-2Section}

{ 

Stationary Mean-Field Games (MFGs), such as Problem~\ref{problem:1}, are often interpreted as equilibrium configurations of long-duration games and are frequently derived from ergodic formulations involving long-term average costs~\cite{ll1}. 
In the present work, however, the underlying control problem is an \emph{optimal stopping (exit-time)} problem: the time horizon is the endogenous exit time \(T_E\), rather than a fixed terminal time \(T\). As a consequence, an ergodic long-time-average reformulation is not natural in this setting. Instead, the stationary MFG system is interpreted as the \emph{steady-state (time-independent equilibrium)} of the associated time-dependent exit-time mean-field game. 

Below, we provide a heuristic derivation showing how the PDE system~\eqref{eq:MFG}--\eqref{eq:BC} formally arises as the stationary state of a coupled time-dependent optimal control and transport problem, following standard MFG modeling frameworks (see, e.g.,~\cite{GPV}).

For simplicity, we focus on the special case where:
\[
    H(x,p) = \frac{1}{2}|p|^2, \quad g(m) = m.
\]
The corresponding Lagrangian is:
\[    L(v) = \sup_{p} \{-p \cdot v - \frac{1}{2}|p|^2\} = \frac{1}{2}|v|^{2}.
\]
We consider the setting of Problem~\ref{problem:1}, namely
a domain $\Omega \subset \mathbb{R}^d$ with boundary $\Gamma = \overline{\Gamma}_D \cup \overline{\Gamma}_N$, let $j \geq 0$ be a $C^1$ function defined on $\Gamma_N$, and $\psi$ a continuous function defined on $\Gamma_D$.

\subsubsection{The Optimal Control}

Consider a population whose density \( m(t,x) \) is known for all time \( t>0 \), and consider a typical player in this population whose trajectory \( \mathbf{x} \) is governed by the following dynamics:
\begin{equation}
 \label{dyn}
    \begin{cases}
			\dot{\mathbf{x}}(\tau) = \mathbf{v}(\tau,\mathbf{x}(\tau)) & \text{for } t< \tau< T_E \\
			\mathbf{x}(t) = x &\mathbf{x}(T_E) \in \Gamma_D.
		\end{cases}
\end{equation}
The effect of the population on the player is two-fold. First, the player is allowed to choose \( \mathbf{v} \) only from the admissible set
\[
    \mathcal{V} = \{ \mathbf{v} \in C^0([t,+\infty)\times \widebar{\Omega},\R^d) \colon -m(t,x)\mathbf{v}(t,x)\cdot \nu = j(x) \text{ on  } \Gamma_N, \enskip m\mathbf{v} \cdot \nu \geq 0 \text{ on  } \Gamma_D\},
\]
which enforces the incoming flow condition \( j \) and prevents (normal) entry through \( \Gamma_D \).
Second, the player optimizes the control according to the following optimal control problem.
\begin{equation}\label{eq:minimize_v}
\underset{ \mathbf{v} \in \mathcal{V} }{ \text{Minimize} } \quad \mathcal{J}[\mathbf{v}; t,x] := \int_{t}^{T_E} 
\left(
\frac{1}{2} |\mathbf{v}(\tau,\mathbf{x}(\tau))|^{2} + m(\tau,\mathbf{x} (\tau)) \right)\d \tau
 + \psi(\mathbf{x}(T_E)),
\end{equation}
where the exit time \( T_E \) is chosen by the player such that
 $\mathbf{x}(T_E)\in \Gamma_D$.
Under enough smoothness assumptions, the classical theory of optimal control ensures that the \emph{value function}
\begin{equation*}
	u(t,x) := \min_{\mathbf{v}} \mathcal{J}[\mathbf{v};t,x],
\end{equation*}
solves the Hamilton-Jacobi equation
\begin{equation} \label{eq:derivedHJ1A}
 -\del_t u(t,x) + \frac{1}{2}| D  u(t,x)|^2 - m(t,x)= 0 \quad \text{ for } t \in (0,\infty), \,  x \in  \Omega.
\end{equation}
Furthermore, it also gives us the optimal velocity field in the feedback form
\begin{equation} \label{eq:derivedOC1A}
	\mathbf{v}^* := - D  u \in \mathcal{V}.
\end{equation}

In the optimal control problem \eqref{eq:minimize_v}, the trajectories \(\mathbf{x}(t)\) evolve within \(\bar{\Omega}\) according
to \eqref{dyn}. When a player reaches the boundary in a transversal direction---i.e., the velocity has a strictly positive normal component--- the player is forced to exit immediately. In this case, \( \mathbf{v}(T_E, \mathbf{x}(T_E))\cdot \nu>0 \) and the player must incur the exit cost \( u(T_E,\mathbf{x}(T_E))= {\psi(\mathbf{x}(T_E))} \).
If a player reaches the boundary tangentially, they are not compelled to exit immediately; instead, they may continue along the boundary—potentially traversing parts of
\(\Gamma_N\) where \(j=0\), and choose a later optimal exit time,
$T_E$.
Because tangential motion does not mandate an immediate exit, the boundary conditions must not force the value function to exactly match the exit cost, rendering a strict Dirichlet boundary condition on \(\Gamma_D\) unsuitable. Instead, we impose the \emph{relaxed} boundary condition
\[
  u(t,x) \;\le\; \psi(x),
  \quad x \,\in\, \Gamma_D,
\]
thereby accounting for the possibility of the absence of an enforced exit at those points.

Thus, $u$ solves \eqref{eq:derivedHJ1A} with the boundary conditions:
\begin{equation} \label{eq:derivedBC1}
 \begin{cases}
     m(t,x) D  u(t,x)\cdot \nu(x) = j(x) & \text{ for } t \in (0,\infty), \,  x \in\Gamma_N\\
       u(t,x)\leq { \psi(x)},
       \quad { m(t,x) } D  u(t,x)\cdot \nu(x) \leq 0 & \text{ for } t \in (0,\infty), \,  x \in\Gamma_D \\
    { (\psi(x) - u(t,x))} { m(t,x) }  D  u(t,x)\cdot \nu(x) =0 & \text{ for } t \in (0,\infty), \,  x \in\Gamma_D
 \end{cases}
\end{equation}
Here, the normal derivative is weighted by the density \(m(t,x)\) to ensure that the flux condition is enforced only at boundary points where agents are present. In particular, if \(m(t,x)=0\) at a point on the boundary, the normal velocity is immaterial, and the flux condition is automatically satisfied.

\subsubsection{Transport}

The second part of the modeling concerns the
collective dynamics of the population. If all agents act optimally, their trajectories will follow the optimal velocity field \eqref{eq:derivedOC1A}. Thus, the initial population density \( m(0,x) = m_0 \) is transported by the flow \( \Phi \) of 
the vector field \( - D  u(t,x) \); that is, 
the flow field of the ODE
\[
    \dot{\mathbf{x}}(t) = - D  u(t,\mathbf{x}(t)) \qquad \text{for } t>0.
\]
Here, for this discussion, we assume that $Du$ is extended to $\R^d$ and the trajectories are stopped at the first time they exit $\Omega$. Hence, for each time $t$ the flow maps $\Omega$ to $\bar \Omega$. 
More explicitly, we define
\[
    \int_{\Omega} \varphi(x) \, m(t,x) \dx = \int_{\Omega} \varphi(\Phi (t,y)) \, m_0(y) \dy.
\]
Differentiating with respect to \( t \), we obtain
\[
    \begin{split}
        \int_{\Omega} \varphi(x) \, \del_t m(t,x) \dx &=
        \int_{\Omega} \del_t \Phi(t,y) \cdot D\varphi(\Phi (t,y))\, m_0(y) \dy\\
		&= -\int_{\Omega}  D  u(t,\Phi(t,y)) \cdot D\varphi(\Phi(t,y)) \, m_0(y) \dy.
    \end{split}
\]
Applying the definition of $m(t,\cdot)$ again, we obtain
\[
        \int_{\Omega} \varphi(x) \, \del_t m(t,x) \dx = -\int_{\Omega}  D  u(t,x) \cdot D\varphi(x)\, m(t,x) \dx,
\]
for all \( \varphi \in C_c^{\infty}(\Omega)\).
This identity is the weak formulation of the transport equation:
\begin{equation} \label{eq:derivationMassTransportA}
	\del_t m(x,t) - \Div( D  u(t,x) \, m(t,x)) = 0.
\end{equation}

\subsubsection{Stationary State}
In the stationary case,  the value function \(u(x)\) and the density \(m(x)\) become time-independent, and the Hamilton–Jacobi equation \eqref{eq:derivedHJ1A} reduces to
\[
\frac{1}{2}| D  u(x)|^2 - m(x) = 0,
\]
while the transport equation \eqref{eq:derivationMassTransportA} becomes
\[
-\Div\big(m(x) D  u(x)\big) = 0.
\]
Together with the boundary conditions \eqref{eq:derivedBC1}, these equations are the stationary mean field game system described in Problem~\ref{problem:1}.
}

\subsection{Background and prior work}
{ 
The MFG model \eqref{eq:MFG}--\eqref{eq:BC} falls within the class of differential games modeling the distribution of large populations of competing rational agents. While the concept of a continuum of agents originated in economics with Aumann~\cite{aumann} and was developed further through anonymous sequential games~\cite{ASG,berginAnonymousSequentialGames1992}, these early discrete-time or static formulations remained largely separate from the PDE community (see \cite{Caines2017}). The modern, continuous-time MFG framework emerged independently around 2006 through two distinct but complementary approaches. Lasry and Lions introduced a PDE-based formulation centered on the coupling of Hamilton--Jacobi and Fokker--Planck equations~\cite{ll1,ll2,lasryMeanFieldGames2007}. Simultaneously, Huang, Malham{\'e}, and Caines developed the Nash Certainty Equivalence principle from a stochastic control perspective, focusing on the limit of large finite populations~\cite{huangLargePopulationStochastic2006,Caines2}. In both frameworks, the core mechanism involves a single representative agent minimizing a cost functional that depends on the statistical distribution of the entire population.

In these games, the model is governed by two primary unknowns: the value function $u$, representing the optimal cost for a typical agent, and the population density $m$. As seen in our model~\eqref{eq:MFG}, these quantities are determined by a coupled system comprising a Hamilton--Jacobi equation, which characterizes the optimal control, and a transport or Fokker--Planck equation governing the density evolution. While the Hamilton--Jacobi equation may be of first or second order, we focus here on first-order models arising from deterministic agent dynamics. The first-order problem with local coupling was analyzed in depth in \cite{cardaliaguetMeanFieldGames2014}, who established the existence and uniqueness of weak solutions under periodic boundary conditions using a variational approach (see also~\cite{graber2017sobolev} for subsequent Sobolev regularity results). Crucially, this variational structure implies that the MFG behaves as a potential game, where the PDE system emerges as the first-order optimality condition of a global optimization problem \cite{lasryMeanFieldGames2007}.

Second-order models, distinguished by the presence of a Laplacian or elliptic term, govern scenarios where agent states are random processes driven by Brownian noise (see~\cite{GPV} for a detailed exposition).  These stationary games are rigorously interpreted as the equilibrium states of long-duration interactions, often derived from ergodic models focusing on the optimization of long-term average costs~\cite{ll1}.

Lasry and Lions initially formulated MFGs under periodic boundary conditions~\cite{ll1,ll2}. Subsequent research has adapted these models to bounded domains, employing either Neumann or Dirichlet conditions to address specific applications. In the first-order setting,  \cite{FGT1} used monotonicity methods to establish weak solutions for the Dirichlet problem.  However, pure Dirichlet formulations cannot strictly prescribe agent inflow, potentially allowing boundary segments to function ambiguously as both entrances and exits depending on the internal dynamics. Despite this limitation, Dirichlet conditions have been successfully applied to specific scenarios such as pedestrian evacuation~\cite{MR4732741} and optimal stopping problems~\cite{bertucciOptimalStoppingMean2018}. For second-order systems, Neumann conditions have been analyzed by Cirant~\cite{cirant} and M{\'e}sz{\'a}ros and Silva~\cite{AMFS}. Notably, the latter work, along with studies on density constraints~\cite{San12,San16}, employs variational techniques to derive the MFG system from optimality conditions. 

}

While prior research has separately addressed MFGs with either first-order Hamilton--Jacobi equations or isolated Dirichlet/Neumann boundary conditions, no study has yet formulated a model that simultaneously incorporates prescribed Neumann inflow and Dirichlet exit constraints. Our work fills this gap by developing a rigorous variational framework that guarantees the existence (and partial uniqueness) of solutions under explicit structural assumptions.

 \subsection{Main results}

This paper demonstrates that incorporating nonstandard mixed boundary conditions, assigning one part of the boundary for agent entry (Neumann conditions) and another for agent exit (relaxed Dirichlet conditions), results in a well-defined stationary MFG formulation. Under minimal regularity assumptions, we introduce a notion of weak solutions (Definition~\ref{def:WeakSolution}) consistent with prior literature and accommodating these mixed boundary conditions.

Building upon this concept, we develop a variational formulation that naturally incorporates the desired boundary behaviors. This approach allows us to:
\begin{itemize}
\item Establish the existence of solutions;
\item Prove a partial uniqueness result (uniqueness of the gradient where the density is positive);
\item Identify the minimal regularity conditions required for well-posedness and, in particular, provide a meaning for the Neumann boundary conditions.
\end{itemize}

In Section~\ref{sec:Assumptions}, we detail our technical assumptions,
define our weak solution framework, and outline formal PDE estimates that motivate our choice of functional spaces. A critical technical challenge is rigorously interpreting the Neumann boundary conditions, as $m D_pH(x, Du)$ is only guaranteed to be in a Lebesgue space. However,  because it is divergence-free, we can interpret the Neumann boundary conditions through the normal trace (see Theorem~\ref{thm:mainNormalTrace}).

To illustrate the limitations of standard Dirichlet conditions, Section~\ref{sec:examples} discusses concrete examples. Moreover,
additional examples illustrate phenomena such as the formation of regions with vanishing density (see Section~\ref{subse:example-2d}).

We now introduce our variational formulation, which naturally yields the appropriate boundary conditions:
\begin{problem} [\bfseries The variational formulation]\label{problem:2}
Let \(\Omega \subset \mathbb{R}^d\) be a bounded, open, and connected \(C^1\)--domain whose boundary \(\Gamma = \partial \Omega\) is partitioned into disjoint sets \( \Gamma_D \) and \( \Gamma_N \); that is, \( \Gamma=\widebar{\Gamma}_D\cup\widebar{\Gamma}_N \). Let
    \[
    j\colon \Gamma_N\to[0,\infty),\quad \psi\colon \Gamma_D\to\mathbb{R},\quad H\colon \Omega\times\mathbb{R}^d\to\mathbb{R},\quad \text{and}\quad G\colon \mathbb{R}\to\mathbb{R}
    \]
    be given functions such that the \( C^1\)--mapping \( p\mapsto H(x,p) \) is convex and that \( G \) is \( C^1\), convex, and increasing.
Minimize the functional
\begin{equation}\label{eq:Minimization}
	\mathcal{I}[w]:= \int_\Omega G\left( H(x, Dw) \right)\dx - \int_{\Gamma_N} j w \d s,
\end{equation}
{ 
over the admissible set
\[ \mathcal{U}=\{w\in W^{1,\gamma}(\Omega): w(x)\le\psi(x)\text{ for }x\in\Gamma_D\}. \]
The exponent \(\gamma\) and detailed assumptions are given in Section~\ref{sec:Assumptions}.}

\end{problem}

Our main result establishes that solutions to the MFG system (Problem~\ref{problem:1}) exist and that they correspond precisely to the minimizers of Problem~\ref{problem:2}.
To achieve this,
in Section~\ref{sec:VarProblem}, we prove the following existence theorem:
\begin{thm}\label{thm:var_prob_exist}
	Under Assumptions~\ref{assume:data1}-\ref{assume:H} and Assumption~\ref{assume:G}, there exists \(u \in \mathcal{U}\) such that
	\[ \begin{split}
			\mathcal{I}[u] = \inf_{w \in \mathcal{U}} \mathcal{I}[w].
		\end{split}\]
\end{thm}
The proof employs the direct method of calculus of variations, addressing the key challenge of establishing coercivity under mixed boundary conditions. Moreover, this variational framework provides
a justification of the
regularity estimates in Section~\ref{sec:formalEstimates}
and is
crucial for the existence of the normal trace for the flux, thus giving precise meaning to the Neumann boundary condition (see Theorem~\ref{thm:mainNormalTrace}).

Furthermore,
we have the following correspondence theorem:

\begin{theorem}\label{thm:correspondance}
Under the Assumptions~\ref{assume:data1}-\ref{assume:g} and Assumption~\ref{assume:G}, suppose that \( G \) and \( g \) satisfy
    \[
        G'(g(\mu)) = \mu \qquad \text{for all }\, \mu \geq 0.
    \]
    Then, a pair \( (m, u) \) is a weak solution to Problem~\ref{problem:1} in the sense of Definition~\ref{def:WeakSolution} if and only if \(u\) is an admissible minimizer of Problem~\ref{problem:2} and \( m = G'(H(x,Du)) \).
\end{theorem}
The detailed proof of this correspondence appears in Section~\ref{sec:Correspondance}.

Nonetheless, the variational formulation naturally gives rise to the contact-set condition on the boundary \( \Gamma_D \), namely \((\psi-u) m D_pH(x, Du) \cdot \nu = 0\). This condition, combined with the fact that the operator associated with the MFG \eqref{eq:MFG}--\eqref{eq:BC} is monotone, helps us to prove the uniqueness result for the measure $m$, as well as a partial uniqueness result for the gradient of the value function, $u$.

\begin{thm}\label{thm:uniquenessMFG}
    Under Assumptions~\ref{assume:data1}-\ref{assume:g}, any two pairs  \((m,u), \, (\eta, v ) \in L^{\b+1}(\Omega) \times W^{1,\gamma}(\Omega)\) that solve Problem~\ref{problem:1}, in the sense of Definition~\ref{def:WeakSolution}, must satisfy the following.
	\begin{enumerate}
		\item For almost all \(x\in \Gamma_D\), we have that
		      \[ \begin{split}
				      (\psi-u)\, \eta \, D_pH(x,D v )\cdot \nu = (\psi- v ) \, m \, D_pH(x,Du)\cdot \nu = 0.
			      \end{split}\]

		\item The equality \(\eta(x) = m(x)\) holds for almost all \(x \in \Omega\).
		\item { The equality \(Du = D v\) holds for almost all \(x \in E_+(m)=\{ x\in \Omega\,:\, m(x) >0\}\).
}

	\end{enumerate}
\end{thm}

The proofs of these uniqueness results are detailed in Section~\ref{sec:monoton_uniquenessMFG}, while the uniqueness of minimizers for Problem~\ref{problem:2} is addressed in Section~\ref{73}. The paper concludes by discussing the Neumann trace in Appendix~\ref{appendix_A}.

\section{Assumptions, Preliminary Results, and  Weak Solutions} \label{sec:Assumptions}

In this section, we lay out our notation, prescribe the assumptions used to prove our results, and give a suitable definition of weak solutions for Problem~\ref{problem:1}.
These assumptions ensure the well-posedness of the variational formulation and enable us to define the boundary conditions in the MFG rigorously.

\subsection{Notation}

Given an exponent \(\tau>1\), \(\tau'\) denotes its convex (H\"older) conjugate; that is,
     \[
     \dfrac{1}{\tau} + \dfrac{1}{\tau'} = 1.
     \]
A differential operator written without a subscript (e.g., \( Du \) or \( \text{div} F \)) refers to differentiation with respect to the variable \( x \). Differentiation with respect to other variables will be indicated by a corresponding subscript. For example, \( D_p H(x, p) \) denotes differentiation with respect to the second argument, \( p \), even when $p$ is replaced by \( Du \).
For a quantity $f$, \( f_+ \) and \( f_- \) are the positive and negative parts of $f$.

Throughout, we adopt the convention that any inequality involving a constant $C$ holds for any sufficiently large $C>0$ independent of the free variables; thus, replacing $C$ with a larger constant does not affect the validity of the inequality.
If the constant depends on any free variable (e.g., \( m \)), this dependence is indicated with a suitable subscript (e.g., \(C_m\)) while maintaining the same interpretation.
This convention minimizes unnecessary notation and allows for \( C \) to vary from one line to another.

\subsection{Assumptions}

We begin by introducing the exponents we use to formulate our main assumptions and results. Assumption \ref{assume:data1} concerns several parameters of the problem, while Assumptions \ref{assume:j} and  \ref{assume:psi}  address the regularity of the boundary data $j$ and $\psi$.
Assumptions \ref{assume:H} and \ref{assume:g} are typical assumptions in this setting on the growth of $H$, its derivatives, and $g$. Finally,
Assumption \ref{assume:G} concerns the function $G$ in the variational problem, Problem \ref{problem:2}. The correspondence between Assumption \ref{assume:g} and \ref{assume:G} is examined in Subsection \ref{consistency}.

\begin{assumption}\label{assume:data1} Our assumptions are formulated in terms of two key exponents:
    \begin{equation}\label{eq:exp_a_b}
    \alpha > 1 \, \text{ and } \, \beta > 0.
    \end{equation}
    To simplify the presentation, we introduce the auxiliary exponent:
    \begin{equation}\label{eq:exp_gamma}
        \gamma := \dfrac{\beta + 1}{\beta} \alpha.
    \end{equation}
\end{assumption}
The exponent $\gamma$ is the Sobolev exponent for 
$u$; its conjugate is used for boundary data $j$. 
The data in the Neumann and Dirichlet boundary conditions satisfy the following assumptions.
\begin{assumption}\label{assume:j}
	The incoming flow \( j \in  L^{\gamma'}(\Gamma_N)\) is nonnegative, and does not vanish identically; that is, \( j\geq 0 \) and \( j\not \equiv 0 \) on $\Gamma_N$.
\end{assumption}

This assumption ensures that the integral $\int_{\Gamma_N} j u \, ds$ is well defined (by the $L^{\gamma'}_{\Gamma_N}$ integrability) and that $j$ is nonnegative (implying that $\Gamma_N$ functions as an entry region).

\begin{assumption}\label{assume:psi}
{ We assume that the exit cost defined on \(\Gamma_D\) admits an extension,
also denoted by \(\psi\), such that \(\psi \in W^{1,\gamma}(\Omega)\).}

\end{assumption}
\begin{remark*}

\begin{enumerate}[label=\tiny$\diamond$, left=0pt ]
     \item The assumption that \( \psi \) is defined on the whole domain $\Omega$ is useful for establishing a lower bound for the variational functional.

     \item By standard trace theorems, this assumption
     { implies that the trace of \(\psi\) on the boundary belongs to
\(W^{1-\frac{1}{\gamma},\gamma}(\Gamma)\), which imposes minimal regularity for the well-posedness of our problem.
}

\end{enumerate}
\end{remark*}

The following assumption pertains to the Hamiltonian \( H \), which governs the interior dynamics of the MFG (Problem~\ref{problem:1}) and also appears in the variational formulation (Problem~\ref{problem:2}).

\begin{assumption}\label{assume:H}
The Hamiltonian \(H:\widebar{\Omega} \times \mathbb{R}^d\to \R\) satisfies the following conditions.
    \begin{enumerate}[label = \normalfont\roman*., ref=\theassumption.\roman*]

		\item\label{assume:H_i} The maps \((x,p)\mapsto H(x,p)\) and \((x,p)\mapsto D_pH(x,p)\) are measurable in \( x \) and continuous with respect to \( p \).

		\item\label{assume:H_ii} For almost every \(x \in \Omega\), the map \(p \mapsto H(x,p)\) is strictly convex.

		\item\label{assume:H_iii} There is a constant \(C>1\) such that, for almost all \(x \in \Omega \) and all \(p \in \mathbb{R}^d\),
        \[
            C^{-1}\,|p|^{\a}-C \leq H(x,p) \leq  C |p|^\a+C.
        \]
        \item\label{assume:H_vi_Now} There is a constant \(C>1\) such that, for almost all \(x \in \Omega \) and all \(p \in \mathbb{R}^d\),
        \[
		  | D_pH(x,p)| \leq C(|p|^{\a-1} + 1).
        \]

        \item\label{assume:H_vi} There is a constant \(C>1\) such that, for almost all \(x \in \Omega \) and all \(p \in \mathbb{R}^d\),
        \[
            - H(x,p) + D_pH(x,p) \cdot p  \geq C^{-1}\,|p|^{\a} -C.
        \]
	\end{enumerate}
\end{assumption}
\begin{remark*}
\begin{enumerate}[label=\tiny$\diamond$, left=0pt ]
    \item These conditions ensure that the integrands appearing in our treatment are well-defined and properly contained in suitable \( L^p \)-spaces.

    \item Assumption~\ref{assume:H_vi} provides a link to the underlying optimal control problem, whose Lagrangian is retrieved by the Legendre transform
    \[
        L(x,v) = \sup_{p} \big( -v \cdot p - H(x,p) \big).
    \]
    When the control \( v \) is optimal, it has the form \( v^* = -D_p H(x,p) \). { Hence,
    \[
        L(x,v^*) = D_p H(x,p) \cdot p - H(x,p),
    \]
    is coercive by Assumption \ref{assume:H_vi}.}

    \item An \textbf{example} of a Hamiltonian that satisfies the assumptions is
    \[
        H(x,p) = (|p|^2 + 1)^{\alpha/2} + 1.
    \]
\end{enumerate}
\end{remark*}

The following assumption pertains to the coupling term \( g \), which appears in the MFG system (Problem~\ref{problem:1}) as a \emph{density-dependent potential} that deters agents from high-density regions.

\begin{assumption} \label{assume:g} The coupling term \( g \) satisfies the following conditions.
    \begin{enumerate}[label = \normalfont\roman*., ref=\theassumption.\roman*]

		\item The map \( g :[0,\infty)\to \R \) is continuous and strictly increasing.
		\item\label{assume:g_ii} There is constant \( C>1 \) such that
		  \[
		      \frac{m^{\b }}{C} -C \leq g(m) \leq  C m^{\b} + C \quad \text{ for all}\enskip m \in [0,\infty) .
		  \]
	\end{enumerate}
\end{assumption}
\begin{remark*}
\begin{enumerate}[label=\tiny$\diamond$, left=0pt ]
     \item  These conditions ensure that the integrands appearing in our treatment are well-defined and properly contained in suitable \( L^p \)--spaces.

     \item  The strict monotonicity assumption makes \( g \) invertible, allowing us to establish the variational formulation of the MFG problem.

     \item  An \textbf{example} of a coupling term that satisfies the assumptions is
     \begin{equation} \label{eq:Example:g}
        g(m)= (m+1)^{\beta}.
     \end{equation}
\end{enumerate}
\end{remark*}

For the treatment of the variational problem, it is convenient to express our assumptions explicitly in terms of \( G \). The following assumptions are equivalent to Assumption~\ref{assume:g} above due to the Definition~\ref{def:extInverse}, which we state after the assumption.
This equivalence is established in Lemma~\ref{lem:EquivgG}.

\begin{customAssumA}{5'}\label{assume:G}
	There is a constant $z_0 \in \R$ such that \( G \) satisfies the following conditions.
    \begin{enumerate}[label = \normalfont\roman*., ref=5'.\roman*]
		\item\label{assume:G_i} The map \( G \) is constant for all \( z \leq z_0 \); that is, \( G(z) = G(z_0) \) for all \( z \leq z_0 \).
		\item\label{assume:G_ii} The map \( G :\R \to \R \) is continuously differentiable.
        \item\label{assume:G_iii} The map \( G \) is strictly convex on \( [z_0, \infty) \).
        \item \label{assume:G_iv} The map \( G \) is strictly increasing on \( [z_0, \infty) \).
		\item\label{assume:G_v}  There is a constant \(C>0\) such that, for all \(z \in [z_0, \infty)\),
        \[
            C^{-1}\,z_+^{1/\b}  - C  \leq G'(z) \leq C \, z_+^{1/\b} + C.
        \]
	\end{enumerate}
\end{customAssumA}
\begin{remark*}
\begin{enumerate}[label=\tiny$\diamond$, left=0pt]
     \item Note that Assumption~\ref{assume:G_v} implies that there is a constant \( C>0 \) such that, for all \( z \in [z_0, \infty) \),
        \begin{equation}\label{eq:ImplicOfassume:G_v}
            C^{-1}\, z_+^{(\b+1)/\b} - C  \leq G(z) \leq  C \,z_+^{(\b+1)/\b} + C.
        \end{equation}

     \item Convexity and lower-boundedness, which we obtain from Assumption~\ref{assume:G_i} and \eqref{eq:ImplicOfassume:G_v}, respectively, are standard assumptions to ensure the existence of a minimizer for the functional \eqref{eq:Minimization}.

     \item An \textbf{example} of a map \( G \) that satisfies the assumptions is
     \begin{equation} \label{eq:Example:G}
        G(z) = \begin{cases}
            \frac{\b}{\b+1} z^{(\b+1)/\b}-z+c & z \geq 1 \\
            c-\frac{1}{\b+1} & z < 1,
            \end{cases}
     \end{equation}
    for some arbitrary constant $c\in \R$.
\end{enumerate}
\end{remark*}

\subsection{
{ Consistency between Assumptions \ref{assume:g} and \ref{assume:G}}
}

\label{consistency}

The equivalence between Assumption~\ref{assume:g} and Assumption~\ref{assume:G} is crucial to ensure the
correspondence between Problem~\ref{problem:1} and Problem~\ref{problem:2}
given in Theorem~\ref{thm:correspondance}.
This equivalence depends on the relation
\begin{equation}\label{eq:G'ODEginv}
    G'(z) = g^{-1}(z),
\end{equation}
where \(  g^{-1}(z) \) is the \emph{extended inverse} of \( g \) defined in the following.

\begin{definition} \label{def:extInverse}
	Let \(g:[0,\infty) \to \R \) satisfy Assumption~\ref{assume:g}. We say that \(g^{-1}:\R\to [0,\infty)\) the \emph{extended inverse} of \( g \) if,
    \begin{enumerate}[label=\roman*.]
        \item\label{subdef:extInverse_i} \( g(g^{-1}(z))=z \) \enskip \enskip for all \( z > g(0) \), and
        \item\label{subdef:extInverse_ii}  \( g^{-1}(z) = 0 \) \enskip \enskip for all \( z \leq g(0) \).
    \end{enumerate}
\end{definition}

\begin{remark}
\begin{enumerate}[label=\tiny$\diamond$, left=0pt]
     \item  Note that the exponent governing the growth of \( G \) is \( (\b+1)' = (\b+1)/\b\).

     \item Note also that \( z_0 =g(0)\) in Assumption~\ref{assume:G}.
\end{enumerate}
\end{remark}

\begin{lemma}\label{lem:EquivgG}
    Let \( G \) be defined according to \eqref{eq:G'ODEginv}. Then, Assumption~\ref{assume:g} and Assumption~\ref{assume:G} are equivalent.
\end{lemma}
\begin{proof} \textbf{Step 1 (\ref{assume:g} \( \implies\) \ref{assume:G}):} First, suppose that \( g \) satisfies Assumption~\ref{assume:g}, and set \( z _0 = g(0)\).
    \begin{enumerate}[label= \roman*.]
        \item By Definition~\ref{def:extInverse}, we have that \( G'(z) = g^{-1}(z) = 0 \) for all \( z \leq z_0 \). Hence, \(G\) is constant on \((-\infty,z_0]\). In particular,
                \[
                    G(z) = G(z_0) \quad \text{ for all }\, z \leq z_0.
                \]
        \item Because \( G'= g^{-1} \in C^0(\R) \), we have \( G \in C^1(\R) \).

        \item Because \( g \) is increasing, its inverse \( G' \) is also increasing on \( [z_0,\infty) \). Therefore, \( G \) is convex.

        \item Because \(g\bigl([0,\infty)\bigr) = [z_0,\infty)\), for every \(z > z_0\) there is an \(m > 0\) such that \(z = g(m)\). Consequently,
        \[
            G'(z) = G'\bigl(g(m)\bigr) = m > 0.
        \]
        Hence, \(G\) is strictly increasing on \(\bigl[z_0,\infty\bigr)\).

        \item\label{subProof:Step1:v} For an arbitrary \( z > z_0 \), substituting \( m = G'(z)\) into Assumption~\ref{assume:g_ii} gives us
        \begin{equation}\label{eq:Previ:lmmaGg1}
            \frac{ G'(z)^{\b }}{C} -C \leq g( G'(z)) = z \leq  C \,G'(z)^{\b} + C.
        \end{equation}
        The left-hand side gives us that
        \[
            G'(z) \leq ( Cz+C)^{1/\b} \leq ( Cz_++C)^{1/\b} \leq C z_+^{1/\b}+C.
        \]
        Notice that \(  C \,  G'(z)^{\b} + C \geq 0 \). Thus, the right-hand side in \eqref{eq:Previ:lmmaGg1} gives us that
        \[
            z_+^{1/\b}  \leq  \left(C \, G'(z)^{\b} + C\right)^{1/\b}  \leq C \, G'(z) + C.
        \]
         The two inequalities above imply that
         \begin{equation}\label{eq:Previ:lmmaGg2}
             \frac{z_+^{1/\b}}{C} -C  \leq G'(z) \leq C z_+^{1/\b}+C.
         \end{equation}
    \end{enumerate}

    \textbf{Step 2 (\ref{assume:G} \( \implies\) \ref{assume:g}):} To establish the converse, suppose that \( G \) satisfies Assumption~\ref{assume:G} and note that \( g = (G')^{-1} \) immediately gives us that \( g(0) = z_0 \).

    \begin{enumerate}[label= \roman*.]
        \item Because \( G' \) is continuous, coercive, and strictly increasing on \( [g(0),\infty) \), its inverse \( g \) is continuous and strictly increasing on \( [0,\infty) \).

        \item For an arbitrary \( m \geq 0 \), substituting \( z = g(m)\) into Assumption~\ref{assume:G_v} gives us
        \[
            \frac{(g(m)_+)^{1/\b}}{C} - C  \leq G'(g(m)) = m \leq C(g(m)_+)^{1/\b} + C.
        \]
        With some manipulations similar to the above, we obtain
        \[
            \frac{m^{\b}}{C} - C  \leq g(m) \leq C m^{\b} + C.
        \]
    \end{enumerate}
    This concludes the proof.
\end{proof}

\subsubsection{\texorpdfstring{An auxiliary bound for \( G \).}{An auxiliary bound for G.}}

The following lemma enables us to calculate the Gateaux derivative in Section~\ref{sec:Correspondance}.
\begin{lemma}\label{lem:boundOnDiffGepG}
    Suppose that Assumptions~\ref{assume:H} and~\ref{assume:G} hold, and let \( p_0, p_1 \in \R^d \) be arbitrary. There is a constant \( C >0 \), independent of \( p_0 \) and \(p_1 \) such that
	\[
        \frac{1}{|\ep |}\lvert G\left( H(x,  p_0+\ep  p_1) \right) - G\left( H(x, p_0) \right)\rvert \leq C  |p_0|^{\gamma} + C |p_1|^{\gamma}+C
	\]
    for all $\ep \in (-1,1)$.
\end{lemma}
\begin{proof}
    Let \( \ep \in (-1,1)\) be arbitrary. By the mean value theorem, for every $x \in \Omega$, there is a $\theta_{x,\ep} \in (0,1)$ such that
    \begin{equation}\label{eq:boundOnDiffGepG1.1}
        \begin{split}
            \frac{1}{\ep } \big[G(H(x, p_0+\ep p_1)) &- G( H(x, p_0))\big] = G'(H(x, p_0+ \ep \theta_{x,\ep} p_1)) \,\, D_pH(x, p_0+ \ep \theta_{x,\ep} p_1)\cdot p_1.
        \end{split}
    \end{equation}
    Assumption~\ref{assume:H_vi_Now} and Young's inequality give us that
    \begin{equation}\label{eq:boundOnDiffGepG1.2}
    \begin{split}
        |D_pH(x, p_0+ \ep \theta_{x,\ep} p_1)\cdot p_1|
        & \leq C |p_0|^{\a} + C  |p_1|^{\a}.
    \end{split}
    \end{equation}
    Similarly, from the monotonicity of \( G' \), along with Assumption~\ref{assume:G_v} and Assumption~\ref{assume:H_iii}, we deduce

    \begin{equation}\label{eq:boundOnDiffGepG1.3}
        G'(H(x, p_0+ \ep \theta_{x,\ep} p_1)) \leq C|p_0|^{\a/\b}+ C|p_1|^{\a/\b}+C.
    \end{equation}
    By combining \eqref{eq:boundOnDiffGepG1.1}--\eqref{eq:boundOnDiffGepG1.3} and applying Young's inequality, we obtain
    \[
        \frac{1}{|\ep |}\lvert G\left( H(x,  p_0+\ep  p_1) \right) - G\left( H(x, p_0) \right)\rvert \leq C|p_0|^{\a(\b+1)/\b}+ C|p_1|^{\a(\b+1)/\b}+C.
    \]
    We conclude by recalling that \( \gamma = \a(\b+1)/\b \).
\end{proof}

\subsection{Formal estimates} \label{sec:formalEstimates}

\begin{lemma} \label{lem:lower-bound-for-int-ju}
	Suppose Assumptions~\ref{assume:j} and~\ref{assume:psi} hold. Let \(u \in \cU \) be an admissible function in Problem~\ref{problem:2}. Then, there is a constant \( C > 1\), independent of \(u\), such that
    \[
		\int_{\Gamma_N} j  u \d s \leq C \|Du\|_{L^\gamma(\Omega)}+C.
    \]
\end{lemma}
\begin{proof}
	First, using H\"older's inequality and that \((u-\psi)_+ \equiv 0\) on \(\Gamma_D\), we obtain
	\[
    \begin{split}
        \int_{\Gamma_N}j(u-\psi)_+ &\leq \|j\|_{L^{\gamma'}(\Gamma_N)}
        \enskip \|(u-\psi)_+\|_{L^\gamma(\Gamma_N)} \\
        &= C\, \|(u-\psi)_+\|_{L^\gamma(\Gamma_N \cup \Gamma_D)}.
	\end{split}
    \]
	The right-hand side can be further bounded by the trace theorem, giving us
	\[
			\int_{\Gamma_N}j(u-\psi)_+ \leq C  \|(u-\psi)_+\|_{L^\gamma(\Omega)} +  C \|D(u-\psi)_+\|_{L^\gamma(\Omega)}.
    \]
{ 
Moreover, we observe that \((u-\psi)_+ \in W^{1,\gamma}(\Omega)\).
Since \(u \in \mathcal{U}\), the trace of \((u-\psi)_+\) vanishes on \(\Gamma_D\).
Because \(|\Gamma_D|>0\), we can apply the Poincaré inequality for functions with zero trace on
a portion of the boundary. The proof of this inequality is similar to one of the standard Poincaré inequality in \cite{E6}, using the continuity of the trace in $\Gamma_D$.
}

    This gives us
	\[
			\|(u-\psi)_+\|_{L^\gamma(\Omega)} \leq C \|D(u-\psi)_+\|_{L^\gamma(\Omega)}
    \]
    and, hence,
    \[
		\begin{split}
		  \int_{\Gamma_N}j(u-\psi)_+ &\leq  C \|D(u-\psi)_+\|_{L^\gamma(\Omega)} \\
          &\leq  C \|Du\|_{L^\gamma(\Omega)} + C\|D\psi\|_{L^\gamma(\Omega)}.
		\end{split}
    \]
    Finally, because \(j\geq 0\), we have
	\[
        \int_{\Gamma_N}j(u-\psi) \leq \int_{\Gamma_N}j(u-\psi)_+.
	\]
	Combining the inequalities above, we obtain
	\[
		\int_{\Gamma_N}j(u-\psi) \d s \leq C \|Du\|_{L^\gamma(\Omega)} + C\|D\psi\|_{L^\gamma(\Omega)}.
	\]
	Our claim follows immediately.
\end{proof}

{ 
Before proceeding to the rigorous analysis, we derive formal a priori estimates for classical
solutions. Although these derivations are formal, they are essential for motivating the choice
of functional spaces (\(L^{\beta+1}\) for \(m\) and \(W^{1,\gamma}\) for \(u\)) required for the
weak formulation and foreshadow the rigorous bounds established later in
Section~\ref{sec:VarProblem}.
}

\begin{proposition}\label{pro:apriori_estimate_rev_v2}
	Under Assumptions~\ref{assume:data1}-\ref{assume:g}, there is constant \(C >0\), depending only on our data, such that, for any classical solution \((m, u)\) to Problem~\ref{problem:1}, we have
    \[
        \| m \|_{L^{\beta+1}(\Omega)}\leq C, \quad  \| m|Du|^\a \|_{L^1(\Omega)} \leq C, \quad and \quad  \| Du \|_{L^{\gamma}(\Omega)}\leq C.
    \]
\end{proposition}
\begin{proof}
    Because \( (m,u) \) solves the H-J equation in the system \eqref{eq:MFG} pointwise, almost everywhere, we have that
    \[
        m = G'(H(x,Du)).
    \]
    Subsequently, we also have
    \[
                C^{-1}|Du|^{\a/\b}-C \leq m \leq C|Du|^{\a/\b}+C.
    \]
    This inequality provides us with three auxiliary bounds:
    \begin{equation}\label{eq:Aprior1.1}
        |Du|^{\gamma} \leq ( Cm + C)^{\beta+1}\leq Cm^{\beta+1}+C,
    \end{equation}
    \begin{equation}\label{eq:Aprior1.2}
        m^{\beta+1} \leq (C|Du|^{\a/\b}+C)^{\beta+1} \leq C|Du|^{\gamma}+C,
    \end{equation}
    and
    \begin{equation}\label{eq:Aprior1.3}
        C^{-1}m^{\b+1}-C\leq C^{-1}m(m-C)_+^{\b} \leq  m|Du|^{\a},
    \end{equation}
    where the constants appearing above are independent of \( m \) and \( u \).  Subsequently,
    \[
        C^{-1} \| m \|_{L^{\beta+1}(\Omega)}^{\beta+1} -C \leq \|Du\|_{L^\gamma(\Omega)}^{\gamma} \leq C \| m \|_{L^{\beta+1}(\Omega)}^{\beta+1} + C.
    \]

    Now, testing the transport equation in \eqref{eq:MFG} by \( (u-\psi) \) gives us
\begin{equation}\label{eq:Aprior1.4}
\begin{split}
    LHS =\int_\Omega m\, D_pH(x, Du) \cdot D(u-\psi) \dx &=\int_{\Gamma} (u-\psi) \, m\, D_pH(x, Du) \cdot \nu \d s\\
    &=\int_{\Gamma_N} j  (u-\psi) \d s.
\end{split}
\end{equation}
The left-hand side (LHS) is bounded below due to Assumption~\ref{assume:H} and Young's inequality (with \( \ep\)). More precisely, we have
\[
\begin{split}
LHS &\geq  \int_\Omega m\left(D_pH(x, Du)\cdot Du - |D_pH(x, Du)| |D\psi| \right)\dx\\
    &\geq \int_\Omega  m(C^{-1}|Du|^\a-\ep |Du|^{\a}-C_\ep|D\psi|^{\a}-C_\ep) \dx.
\end{split}
\]
Selecting \(\ep \) appropriately gives us
\[
\begin{split}
    LHS &\geq C^{-1} \int_\Omega  m|Du|^\a\dx-C\int_\Omega  m (|D\psi|^{\a} +1)\dx\\
\end{split}
\]
{ 
Using inequality \eqref{eq:Aprior1.3}, the first term is bounded below by
\(C^{-1}\|m\|_{L^{\beta+1}(\Omega)}^{\beta+1}-C\).
For the second term, we apply Hölder's inequality followed by Young's inequality with \(\varepsilon\):
\begin{align*}
C\int_\Omega m (|D\psi|^{\alpha}+1)\,dx
&\le C \|m\|_{L^{\beta+1}}
   \||D\psi|^{\alpha}+1\|_{L^{(\beta+1)'}} \\
&\le \varepsilon \|m\|_{L^{\beta+1}(\Omega)}^{\beta+1}
   + C_\varepsilon \||D\psi|^{\alpha}+1\|_{L^{(\beta+1)'}}^{(\beta+1)'}.
\end{align*}
Combining these estimates, we obtain
}

\[
\begin{split}
    LHS  &\geq C^{-1} \| m \|_{L^{\beta+1}(\Omega)}^{\beta+1} -\ep \| m \|_{L^{\beta+1}(\Omega)}^{\beta+1} -C_\ep \| |D\psi|^{\a} +1\|_{L^{(\b+1)'}(\Omega)}^{(\b+1)'} \\
\end{split}
\]
Selecting \(\ep \) appropriately gives us
\[
\begin{split}
    LHS &\geq C^{-1} \| m \|_{L^{\beta+1}}^{\beta+1}-C \\
\end{split}
\]
Subsequently, \eqref{eq:Aprior1.4} becomes
\[
\begin{split}
    \int_{\Gamma_N} j  (u-\psi) \d s  &\geq C^{-1} \| m \|_{L^{\beta+1}}^{\beta+1} -C \\
\end{split}
\]
On the other hand, by Lemma~\ref{lem:lower-bound-for-int-ju} and Young's inequality, we have
\[
\begin{split}
    C^{-1} \| m \|_{L^{\beta+1}}^{\beta+1} -C &\leq C \|Du\|_{L^\gamma(\Omega)}+C\\
    &\leq \ep  \|Du\|_{L^\gamma(\Omega)}^{\gamma}+C_\ep\leq \ep C\| m \|_{L^{\beta+1}}^{\beta+1}+C\ep
\end{split}
\]
Selecting \( \ep \) appropriately, we obtain
\[
\begin{split}
    \| m \|_{L^{\beta+1}}^{\beta+1} \leq C.
\end{split}
\]
Note that \( C \) is independent of \( m \) and \( u \). The other two estimates follow immediately using \eqref{eq:Aprior1.1} and \eqref{eq:Aprior1.3}.
\end{proof}

\subsection{The normal trace theorem}

Here, we present the normal trace theorem that ensures the well-definedness of the Neumann boundary condition in a suitable dual space. We briefly overview this theorem, deferring detailed proofs and discussion to Appendix~\ref{appendix_A}. To this end, we first define a function space tailored to our setting.
 \begin{definition} \label{def:DivSobolevSpaceMain}
	The \emph{divergence Sobolev space} \( W^{\gamma'}( \Div; \Omega) \) is the space consisting of all vector fields
	\[
        \mathbf{w} \in L^{\gamma'}(\Omega; \R^d) \quad\text{such that}\quad \Div(\mathbf{w}) \in L^{\gamma'}(\Omega),
    \]
	where \( \Div(\mathbf{w}) \) is interpreted in the weak sense. (See Definition~\ref{def:weakDiv} in the appendix). This space is, naturally, equipped with the norm
	\[
        \| \mathbf{w} \|_{W^{{\gamma'}}( \Div; \Omega)}  := \Big( \|\mathbf{w}\|_{L^{\gamma'}(\Omega; \R^d)}^{\gamma'} + \|\Div(\mathbf{w})\|_{L^{\gamma'}(\Omega)}^{\gamma'}\Big)^{1/{\gamma'}}.
	\]
\end{definition}
 Now, we are ready to state the normal trace theorem.
\begin{thm}\label{thm:mainNormalTrace}
	Let \(\Omega\) be an open, bounded, \(C^1\)--domain. Then, there exists a bounded linear operator
	\[
        \tr_1: W^{\gamma'}( \Div; \Omega) \to W^{-1+\frac{1}{\gamma} , \, {\gamma'}}(\Gamma),
	\]
	satisfying the following.
	\begin{enumerate}[label = {\normalfont (\roman*)}]
		\item For all smooth vector fields \(\mathbf{F}\in C^\infty(\widebar{\Omega};\mathbb{R}^d)\),
        \[
            \tr_1(\mathbf{F}) = \mathbf{F}|_{\Gamma}\cdot \nu.
        \]

		\item  There is a constant \( C >0 \) such that, for all \(\mathbf{F} \in  W^{\gamma'}( \Div; \Omega)\),
        \[
        \| \tr_1(\mathbf{F})\|_{W^{-1+\frac{1}{\gamma}, {\gamma'}}(\Gamma)} \leq C\|\mathbf{F}\|_{W^{\gamma'}( \Div; \Omega)}.
		\]
	\end{enumerate}
\end{thm}
\begin{remark}
\begin{enumerate}[label = \roman*., left=0pt ]
    \item The space \( W^{-1+\frac{1}{\gamma}, \, {\gamma'}}(\Gamma) \) is the dual of \( W^{1-\frac{1}{\gamma}, \, {\gamma}}(\Gamma) \), which is the \emph{trace space} of \( W^{1,{\gamma}}(\Gamma) \) \cite{DiBenedetto1}.

    \item Henceforth, we write \( F\cdot \nu = \tr_1(F)  \), where \( \nu \) is the normal to the boundary \( \Gamma \).
\end{enumerate}
\end{remark}
\begin{proof}
    We defer the proof to our discussion in the Appendix~\ref{appendix_A}.
\end{proof}

\begin{remark}
  The estimates from Section~\ref{sec:formalEstimates}, combined with Assumption~\ref{assume:H}, imply that \( m D_p H (x, Du) \) \( \in L^{\gamma'}(\Omega; \mathbb{R}^d)\).
  Since $\Div(mD_pH(x, Du))=0$, we have \( mD_pH(x, Du)\in W^{\gamma'}( \Div; \Omega) \). Thus,
  the previous theorem ensures that the normal trace $mD_pH(x, Du)\cdot \nu$ is well defined in $\Gamma_N$.
\end{remark}

\subsection{Definition of weak solutions}

We define a weak solution by extending the definition employed in \cite{cardaliaguetMeanFieldGames2014} to include additional constraints that rigorously account for boundary data. A key aspect of this definition is that we enforce boundary conditions in the sense of traces.
More precisely, we define a weak solution as follows.
\begin{definition}
	\label{def:WeakSolution}
	A \emph{weak solution} to Problem~\ref{problem:1} is a pair of functions \((m,u) \in L^{\b+1}(\Omega) \times W^{1,\gamma}(\Omega)\) that satisfies the following conditions.
	\begin{enumerate}[label = (C\arabic*)]
		\setcounter{enumi}{-1}
        \item The density $m$ is nonnegative.

        \item The Hamilton-Jacobi equation in \eqref{eq:MFG} holds in the following sense:
            \[
            \begin{split}
                H(x, Du(x))  = g(m(x)), \quad &\text{ a.e. in } \{ x \in \Omega \colon m(x) > 0 \},\enskip \text{ and} \\
                H(x, Du(x))  \leq g(0), \quad &\text{ a.e. in } \{ x \in \Omega \colon m(x) = 0 \}.
            \end{split}
		  \]
        \item{ 
        The Dirichlet boundary condition in \eqref{eq:BC} is satisfied in the trace sense, namely:
\begin{enumerate}
\item the boundary trace of \(u\) is well defined and satisfies
\[
u \le \psi \quad \text{a.e. on } \Gamma_D;
\]
\item the normal trace of the flux, \(m D_pH(x,Du)\), is well defined and satisfies
\[
m D_p H \cdot \nu \le 0, 
\]
in the sense of distributions on \(\Gamma_D\),  i.e., when tested against nonnegative test functions supported on \(\Gamma_D\);
\item the complementarity condition
\[
(\psi- u)(m D_pH(x,Du)\cdot \nu) = 0
\]
 holds in the sense of distributions in $\Gamma_D$.
\end{enumerate}
}
\end{enumerate}

\end{definition}

\section{Examples} \label{sec:examples}

This section presents explicit closed-form examples of MFGs that motivate our model and illustrate characteristic behaviors of solutions to \eqref{eq:1d_MFG}--\eqref{eq:1d_BC}. In particular, we demonstrate that imposing a strict Dirichlet boundary condition in first-order problems may lead to artificial inflows at the boundary, while phenomena such as the formation of empty regions (where \(m \equiv 0\)) and potential non-uniqueness of the value function \(u\) can arise, issues of particular importance in scenarios with congestion or overcrowding. Moreover, these examples serve two key purposes: they provide benchmarks for verifying numerical methods, and they reveal intrinsic connections among the model’s components, thereby facilitating the development of new theoretical results (see, e.g., \cite{MR2928378}).

\subsection{Limitations of Strict Dirichlet Conditions in Exit Problems} \label{subsec:example1111}

In exit problems in optimal control, the system’s behavior within the domain is modeled by a Hamilton–Jacobi equation, while the exit cost is prescribed via a Dirichlet boundary condition. In the context of viscosity solutions,
the boundary value may not be achieved, and the solution becomes discontinuous at the boundary, see \cite{Barlesbook}. A simple example is the Hamilton-Jacobi equation
\[
			\frac12 u_x^2 = 0 \quad  \text{in } (0,1),
\]
subjected to the Dirichlet boundary conditions $u(0)=0$ and $u(1)=1$. The unique viscosity solution for this problem is $u=0$ on $(0,1)$, which is
discontinuous at $x=1$.
{ 
However, in the context of MFGs, we typically seek solutions in Sobolev spaces
(continuous up to the boundary). Strictly imposing the Dirichlet condition forces the solution
to attain the boundary value. If the imposed cost is high relative to the interior dynamics,
the system may induce artificial inflows to satisfy this constraint, leading to unintended
behavior at boundaries intended as exits.
}

To illustrate this behavior, consider the MFG
	\begin{equation}\label{eq:first_order_model}
		\begin{cases}
			\frac12 u_x^2 = m &\text{in } (0,1),  \\[1mm]
			- (m\,u_x)_x = 1      & \text{in } (0,1),
		\end{cases}
	\end{equation}
subject to the Dirichlet boundary conditions
	\begin{equation*}
		u(0)=0 \quad \text{and} \quad u(1)=a.
	\end{equation*}
{ 
We introduce a source term (the “1” on the right-hand side) in this illustrative example.
This contrasts with Problem~\ref{problem:1}, which is source-free. Here, the source term ensures a
non-trivial population density without requiring boundary inflow, allowing us to isolate and
emphasize the effect of the Dirichlet boundary condition on the flux direction.
}

The system is explicitly solvable. Integrating the second equation in \eqref{eq:first_order_model} yields that the flux of agents is
	\begin{equation}\label{eq:Jin1stOE}
		j(x) := -m\,u_x = x + c,
	\end{equation}
for some constant \(c \in \mathbb{R}\). Substituting this expression into the Hamilton–Jacobi equation, \(\frac12 u_x^2 = m\), leads to
	\[
		m = \frac{1}{2^{1/3}} (x+c)^{2/3}.
	\]
Next, substituting back into \eqref{eq:Jin1stOE} and applying the boundary condition \(u(0)=0\) yields
	\[
		u(x) = \frac{3}{2^{5/3}}\Big[c^{4/3} - (x+c)^{4/3}\Big].
	\]
Finally, imposing the boundary condition at \(x=1\) determines \(c\) via
	\begin{equation} \label{eq:aAndcRelation}
		\frac{3}{2^{5/3}}\Big[c^{4/3} - (1+c)^{4/3}\Big] = a.
	\end{equation}

Applying the mean value theorem to the difference \((1+c)^{4/3} - c^{4/3}\) yields
	\begin{equation} \label{eq:aAndcRelationWithMVT}
		a = -2^{1/3}\Big(c+\theta_{c}\Big)^{1/3},
	\end{equation}
with some \(\theta_c \in (0,1)\). Consequently, a large positive \(a\) forces \(c\) to be significantly negative, while a large negative \(a\) results in a correspondingly large positive \(c\).

Now, we examine how variations in \(a\) influence the flow direction at the boundaries. The flux at the boundaries is given by
	\[
		j(0)=c \quad \text{and} \quad j(1)=1+c.
	\]
Recall that the interpretation of these fluxes depends on the outward normal: at \(x=0\), a negative value \(j(0)\) indicates outflow (exit), whereas at \(x=1\), a positive value \(j(1)\) indicates outflow. Thus, if \(a\) is sufficiently large and positive (so that \(c \ll -1\)), then \(j(0)\) is negative—implying an exit at \(x=0\), while \(j(1)\) corresponds to an inflow (entry). Conversely, if \(a\) is very negative, then \(c\) becomes positive, so that \(j(0)\) indicates an entry at \(x=0\) and \(j(1)\) an exit at \(x=1\). A third scenario arises when \(-1 < c < 0\); in this case, the fluxes are such that both boundaries effectively serve as exit points.
The reason for this behavior is the following.
Since the running cost scales with the density of players, imposing an excessively high Dirichlet boundary condition may force the system to achieve an unrealistically high density near the boundary. Consequently, the system may induce an artificial inflow of players to satisfy the boundary requirement, as seen in our example, where a large positive \(a\) resulted in \(x=1\) functioning as an entry point.

In summary, these observations confirm that the Dirichlet boundary conditions do not predetermine the entry-exit regions of the domain boundary; one must solve the problem to determine the actual behavior. This motivates the adoption of a relaxed Dirichlet condition to ensure that exit and entry regions are correctly and a priori identified.

\subsection{One-Dimensional Mean Field Games with Mixed Boundary Conditions}

In contrast to the strict Dirichlet boundary condition from Subsection~\ref{subsec:example1111}, we now consider a simple one-dimensional MFG with a relaxed Dirichlet condition at one end of the interval and a prescribed Neumann inflow at the other.
This example admits an explicit solution that offers insights into the
interplay between inflow constraints, exit costs, and the resulting player density of MFG systems of the form \eqref{eq:MFG}--\eqref{eq:BC}.

Let \( \Omega = (0,1) \), and consider the MFG system
\begin{equation} \label{eq:1d_MFG}
\begin{cases}
\frac{1}{2} u_x^2 + V(x) = m, \\
- (m\,u_x)_x = 0,
\end{cases}
\end{equation}
subject to the boundary conditions
\begin{equation} \label{eq:1d_BC}
\begin{cases}
    -m(0)\,u_x(0) = j_0, \\[.25em]
	u(1) \leq 0, \\[.25em]
    u(1) \, m(1) \, u_x(1) = 0,
\end{cases}
\end{equation}
where the nonnegative constant \(j_0 \) corresponds to the prescribed inflow of players at \(x=0\). The key distinction of this setup is that we impose a relaxed Dirichlet condition at \( x = 1 \), which specifies the players' exit cost but allows for selective departure. This strategic combination of boundary conditions is crucial for the precise modeling of the flow of players within the domain and for ascertaining the influence of exit costs on the players' flow and distribution. Moreover, the term \( V \) is called the \emph{potential}, and it represents the spatial preferences of the players in MFG models (i.e., distribution of attraction sites in the domain of the game).

We divide our analysis of this example into two cases: \(j_0 = 0\) (no inflow) and \(j_0 > 0\) (positive inflow).

\paragraph{Case 1 (\( j_0 = 0 \)):}
In the absence of inflow, the transport equation in \eqref{eq:1d_MFG} and the left boundary condition in \eqref{eq:1d_BC} imply that
\[
m(x) u_x(x) = m(0) u_x(0) = 0 \quad \text{for all } x \in (0,1);
\]
hence, either \( u_x(x) = 0 \) or \( m(x) = 0 \). Taking Definition~\ref{def:WeakSolution} into consideration, we substitute these alternatives into the HJ equation in \eqref{eq:1d_MFG}. Thus, we obtain that either
\[
\begin{array}{ll}
    m(x) >0 \enskip \text{ and } \enskip u_x(x)=0 & \implies   m(x) = V(x),
\end{array}
\]
or
\[
\begin{array}{ll}
    m(x)=0  &\implies  \frac12 u_x(x)^2 \leq -V(x).
\end{array}
\]
The first alternative is only possible if \( m(x) = V(x) = V_+(x) \), while the second alternative can only hold if $ V \leq 0$, implying also that \( m(x)=V_+(x)=0 \).
A particular solution can be obtained by considering the equality case in the Hamilton-Jacobi equation, leading to the solution
\[
\begin{cases}
m(x) = V_+(x), \\[.5em]
u(x) = u(1) \pm \displaystyle \sqrt{2} \int_x^1 \sqrt{V_-(z)}\dz.
\end{cases}
\]
Thus, if $V$ changes sign, the MFG \eqref{eq:1d_MFG}--\eqref{eq:1d_BC} admits at least two distinct solutions for \(u\).

This example illustrates that the density \(m\) may vanish in certain regions. Owing to the invariance of \( u \) under the addition of a constant, we may assume without loss of generality that \(u(1)=0\). In the zero-current case, either the velocity $u_x$ is zero (so that \(m\geq 0\)) or the velocity is nonzero, in which case \(m\) must vanish.

\paragraph{Case 2 (\( j_0 > 0 \)):}
{ 
The transport equation \(- (m u_x)_x = 0\) implies that the flux is constant:
\(m(x)u_x(x)=C\). The boundary condition at \(x=0\) is \(-m(0)u_x(0)=j_0\).
Thus \(C=-j_0\), and we must have
\[
m(x)u_x(x)=-j_0.
\]
}

Because \(j_0 \neq 0\), it follows that \(m(x) > 0\) for all \(x\); hence,
\begin{equation}\label{eq:u_j_m}
	u_x(x) = - \frac{j_0}{m(x)}.
\end{equation}
Substituting this into the first equation of \eqref{eq:1d_MFG} and multiplying by \(m(x)^2\) yields
\begin{equation}\label{eq:m3}
	m(x)^3 - V(x)\,m(x)^2 - \frac{j_0^2}{2} = 0.
\end{equation}
This cubic equation admits three real roots, of which exactly one is positive. Thus, we obtain
\[ \begin{split}
		m(x) =
		\begin{cases}
			\sqrt[3]{ a(x) + b(x)} + \sqrt[3]{ a(x) - b(x)} + \dfrac{V(x)}{3},                                                           & V > \kappa,    \\
			\dfrac{j_0^{2/3}}{2},                                                                                                        & V(x) = \kappa, \\
			 v ^{k_0} \sqrt[3]{ a(x) + b(x)} + \dfrac{1}{ v ^{k_0} \sqrt[3]{ a(x) + b(x)}} \left( \dfrac{V}{3} \right)^2 + \dfrac{V}{3}, & V(x) < \kappa,
		\end{cases}
	\end{split}\]
where \( v \) is a complex cube root of \(1\), \( k_0 = 0, 1, 2 \), and
\[ \begin{split}
		\kappa =  - \dfrac{3}{2}  j_0^{2/3}, \quad
		a(x) = \left(\frac{V}{3}\right)^3 + \left(\frac{j_0}{2}\right)^2, \quad
		b(x) = \frac{j_0}{2} \sqrt{ 2 \left(\frac{V}{3}\right)^3  + \left(\frac{j_0}{2}\right)^2}.
	\end{split}\]
Although we have a complex expression in the case of $V(x) < \kappa$, all three roots of \(m\), for \( k_0 = 0, 1, 2 \), are real due to the coefficients of the cubic equation.
Moreover, two of them are negative, and only one is positive, which we choose.

From \eqref{eq:u_j_m}, it follows that \(u_x(x) < 0\). Consequently, the first equation in \eqref{eq:1d_MFG} implies
\[
u_x(x) = -\sqrt{2\bigl(m(x)-V(x)\bigr)},
\]
where the negative square root is chosen for consistency. Imposing the compatibility condition \(u(1)=0\) from \eqref{eq:1d_BC} yields
\[
u(x) = \int_x^1 \sqrt{2\bigl(m(x)-V(x)\bigr)}\dx.
\]

This example illustrates that, in one dimension, a positive current guarantees that the density \(m\) remains strictly positive—regardless of the sign and amplitude of the potential \(V\). This behavior contrasts with the two-dimensional case, where \(m\) may vanish in certain regions,
as we discuss in Subsection~\ref{subse:example-2d}.

\subsubsection{Confirmation of the variational formulation} It is worth noting that, in this example, the variational formulation that corresponds to the MFG \eqref{eq:1d_MFG}--\eqref{eq:1d_BC} is given by
\begin{equation}\label{eq:1d_var}
	\inf_{u} \int_0^1 \frac{1}{2}\left(\Bigl(\frac{1}{2}u_x^2 + V(x)\Bigr)_+\right)^2 \dx - j_0 u(1),
\end{equation}
subject to the constraint \(u(1) \leq 0\). Here, the coupling term \( g(m) = m \) and
\[
G(z) = \frac{1}{2}(z_+)^2.
\]
This variational problem can be solved numerically via a simple finite difference method in conjunction with Mathematica's built-in function \texttt{FindMinimum}. As a result, one can obtain an experimental confirmation of the correspondence result in Theorem~\ref{thm:correspondance}.

\paragraph{The case \(j_0 = 0\):}
Figure~\ref{figure:j0=0} compares the numerical results with the analytical solution.

\begin{figure}[H]
	\centering
	\includegraphics[width= 0.98\textwidth]{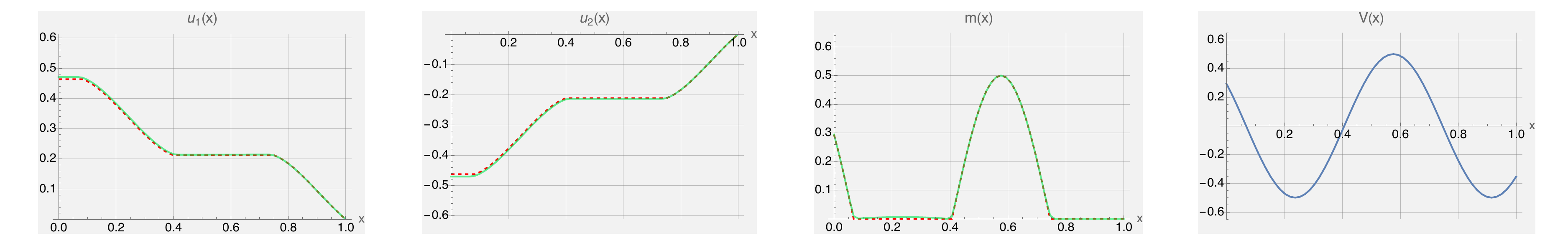}
    \caption{Comparison of the analytical (red dashed) and numerical (green) solutions for \(m\) and two distinct \(u_1\) and \(u_2\), with \(j_0 = 0\), and \(V(x) = 0.5 \sin\bigl(3\pi (x+0.25)\bigr)\).}

	\label{figure:j0=0}
\end{figure}

\paragraph{The case \(j_0 > 0\):}

Figure~\ref{figure:j0g0} compares the numerical results and the analytical solution.
\begin{figure}[H]
	\centering
	\includegraphics[width= 0.95\textwidth]{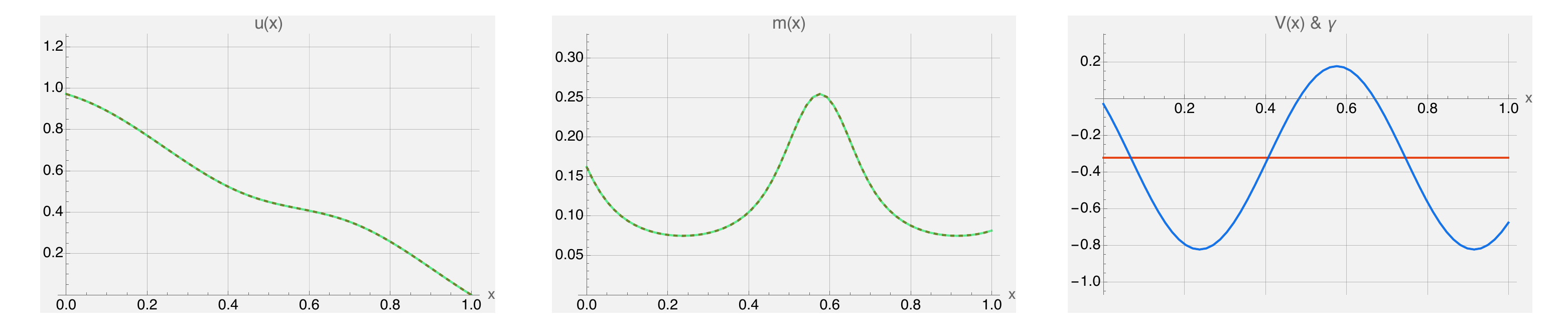}
    \caption{Comparison of the analytical (red dashed) and numerical (green) solutions for  \(u\) and  \(m\), with \(j_0=0.1\)  and  \(V(x)=\gamma + 0.5 \sin\bigl(3\pi (x+0.25)\bigr)\).}

			\label{figure:j0g0}
\end{figure}

\subsection{A two-dimensional example}\label{subse:example-2d}

Let \(\Omega = (0,1) \times (0,1)\) be the unit square whose boundary is partitioned as \( \Gamma = \Gamma_N \cup \Gamma_D\), where
\[
\begin{split}
\Gamma_N &= \{0\} \times [0,1] \cup [0,1] \times \{1\}, \quad \text{and} \\
\Gamma_D &= \{1\} \times [0,1] \cup [0,1] \times \{0\}.
\end{split}
\]
Set
\[
V(x,y) := 3e^{-\pi x}\cos(\pi y) - \frac{1}{2}\pi^2e^{-2\pi x},
\]
and consider the following MFG problem on \(\Omega\):

\medskip
\noindent\textbf{The MFG System:}
\[
\begin{cases}
\frac{1}{2}|D u(x,y)|^2 + V(x,y) = m, & \text{for } (x,y) \in \Omega, \\[1mm]
-\Div(m D u) = 0, & \text{for } (x,y) \in \Omega.
\end{cases}
\]

\medskip
\noindent\textbf{Neumann Boundary Conditions:}
\[
\begin{cases}
-m\,\partial_x u(0,y) = \dfrac{3\pi}{2}\,[\sin(2\pi y)]_+, & \text{for } y \in [0,1], \\[.35em]
m\,\partial_y u(x,1) = 0, & \text{for } x \in [0,1].
\end{cases}
\]

\medskip
\noindent\textbf{Relaxed Dirichlet Boundary Conditions:}
\[
\begin{cases}
u(1,y) \le e^{-\pi}\sin(\pi y), & \text{for } y \in [0,1], \\[.35em]
u(x,0) \le 0, & \text{for } x \in [0,1].
\end{cases}
\]

\medskip
\noindent\textbf{No-Entry Conditions:}
\[
\begin{cases}
m\,\partial_x u(1,y) \le 0, & \text{if } u(1,y)= e^{-\pi}\sin(\pi y) \text{ for } y \in [0,1], \\[.35em]
-m\,\partial_y u(x,0) \le 0, & \text{if } u(x,0)= 0 \text{ for } x \in [0,1].
\end{cases}
\]

\medskip
\noindent\textbf{Contact-Set Conditions:}
\[
\begin{cases}
m\bigl(u(1,y) - e^{-\pi}\sin(\pi y)\bigr)\,\partial_x u(1,y) = 0, & \text{for } y \in [0,1], \\[.35em]
-m\,u(x,0)\,\partial_y u(x,0) = 0, & \text{for } x \in [0,1].
\end{cases}
\]

A direct computation shows that the pair
\[
\begin{cases}
u(x,y) = e^{-\pi x}\sin(\pi y), \\[1mm]
m(x,y) = 3e^{-\pi x}\,[\cos(\pi y)]_+,
\end{cases}
\]
solves the above MFG in the sense of Definition~\ref{def:WeakSolution}.

\begin{figure}[H]
    \centering
    \begin{subfigure}{0.48\textwidth}
        \centering
        \includegraphics[width=\textwidth]{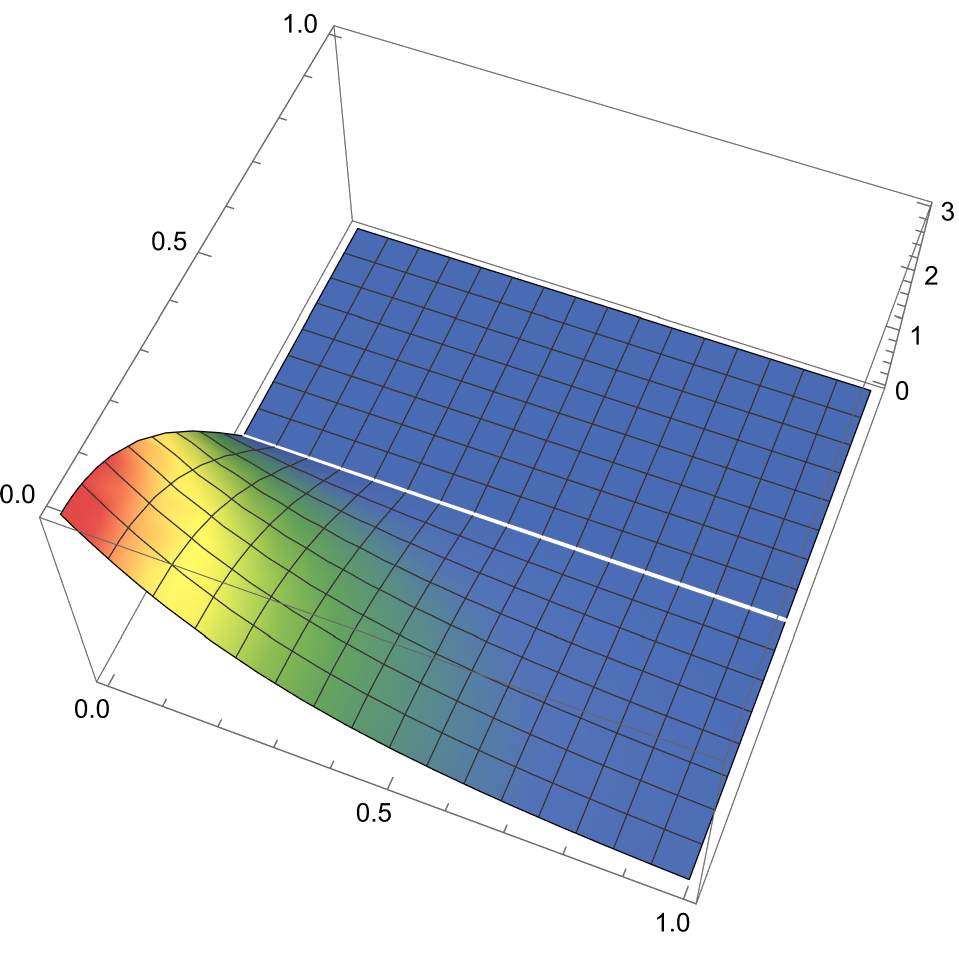}
	\caption{Graph of \(m\).} \label{figure:Eg2GraphOfm}
    \end{subfigure}\quad
    \begin{subfigure}{0.4\textwidth}
        \centering
        \includegraphics[width=\textwidth]{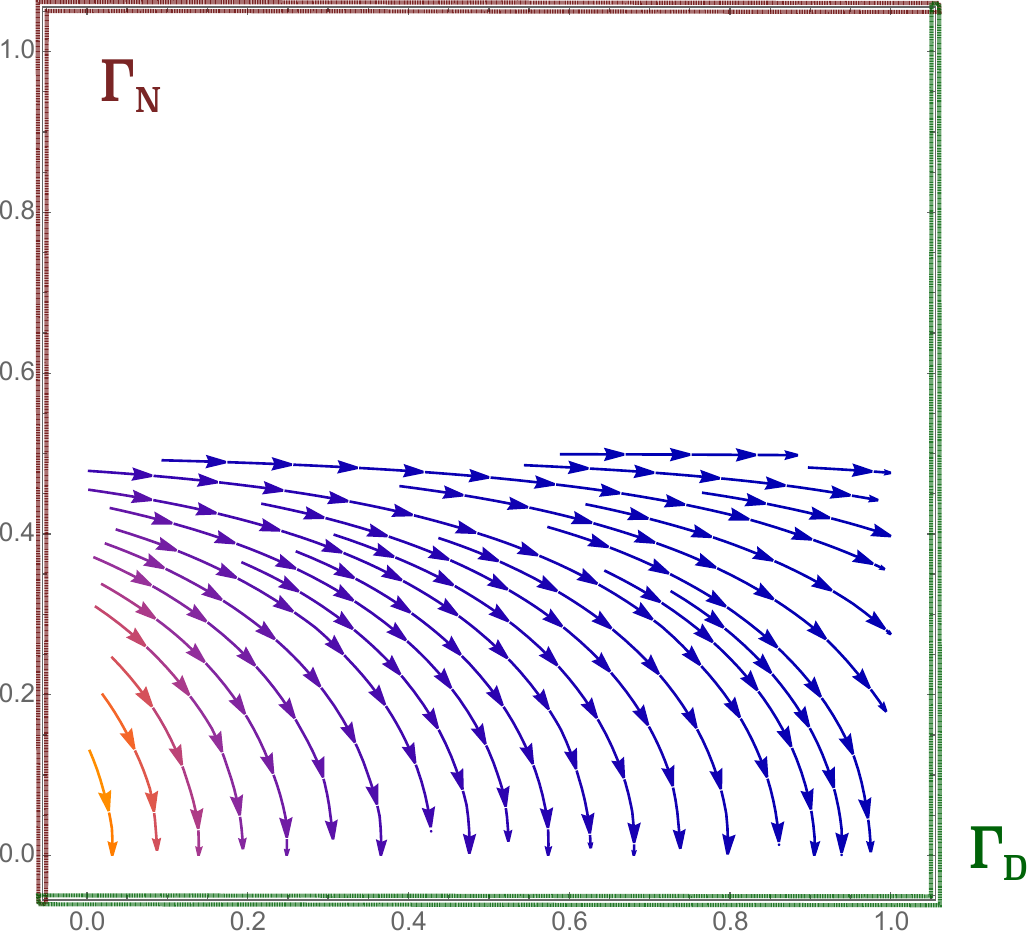}
	\caption{\centering Graph of the flow, \(J\), where the upper half shows the absence of flow in empty regions.} \label{figure:Eg2GraphOfJ}
    \end{subfigure}
    \caption{
    Illustration of MFG Behavior: Density
$m$ and Flux
$J$ Highlighting Empty Regions and Flow Patterns.}
     \label{figure:fullEg2GraphOfJ}
\end{figure}

This example illustrates several intrinsic features of the MFG. We highlight two phenomena relevant to our discussion:

\begin{enumerate}
    \item In the above example, the function \(\psi\) on the Dirichlet boundary is given by
    \[
    \psi(x,y) = \begin{cases}
    e^{-\pi}\sin(\pi y) & \text{for } x = 1,\; y \in [0,1], \\[.35em]
    0 & \text{for } x \in [0,1],\; y = 0;
    \end{cases}
    \]
    Accordingly,
    $u=\psi$ on \(\Gamma_D\). However, the constraints imply
    \[
    \begin{cases}
    m\,\partial_x u(1,y) = -\dfrac{3\pi e^{-2\pi}}{2}\,[\sin(2\pi y)]_+ \le 0, & \text{for } y \in [0,1], \\[.35em]
    -m\,\partial_y u(x,0) = -3\pi e^{-2\pi x}\,[\cos(\pi y)]_+ \le 0, & \text{for } x \in [0,1].
    \end{cases}
    \]
    In particular, since \(m\,\partial_x u(1,y)=0\) for \(y \in [\tfrac{1}{2},1]\), the condition \(u=\psi\) on \(\Gamma_D\) does not necessarily force players to exit the domain in that region.

    \item Although in this example \(u=\psi\) on \(\Gamma_D\), one may modify the boundary datum by setting
    \[
    \widetilde{\psi}(x,y) = \begin{cases}
    e^{-\pi}\sin(\pi y) + \bigl(y-\tfrac{1}{2}\bigr)_+^2 & \text{for } x = 1,\; y \in [0,1], \\[.35em]
    0 & \text{for } x \in [0,1],\; y = 0,
    \end{cases}
    \]
    thus, yielding a modified problem that still admits the same solution. This modification illustrates that a solution \(u\) may differ from \(\psi\) on \(\Gamma_D\), emphasizing the necessity for the relaxation of the Dirichlet boundary condition.
\end{enumerate}

\subsection{More two-dimensional examples with a complex variable construction}
\label{complexmethod}

We introduce a novel method for generating explicit solutions to MFGs using properties of complex-valued functions in two dimensions. This approach enables the construction of solutions that exhibit regions devoid of agents, which is the first time such a method has been proposed in the MFG literature.
The method takes advantage of the fact that any harmonic function in a simply connected domain admits a harmonic conjugate, with the potential \(V\) chosen a posteriori to satisfy the MFG system.

More precisely, we consider MFGs of the form
\begin{equation} \begin{split}
		\label{eq:CompEg}
		\begin{cases}
			\frac12 |D u|^2 +V
			=  m^{1/q}                                              & \text{ for } (x,y) \in \Omega, \\
			-\Div(m D u) = -D m \cdot D u - m\Delta u = 0 & \text{ for } (x,y) \in \Omega.
		\end{cases}
	\end{split}\end{equation}

Consider a differentiable, complex-valued mapping
\[
z = x+iy \mapsto f(z)= u(x,y) + i\,m(x,y),
\]
where \(u\) and \(m\) are real-valued functions. By the Cauchy-Riemann equations, we have
\[ \begin{split}
		\begin{cases}
			\partial_xu  & = \partial_y m     \\
			\partial_y u & = - \partial_x m.
		\end{cases}
	\end{split}\]
Thus,
\[ \begin{split}
		D u \cdot D m = \partial_xu \, \partial_x m + \partial_y u\, \partial_y m = 0.
	\end{split}\]
Because \(f\) is differentiable, \(u\) is harmonic; that is, \(u\) solves \(-\Delta u = 0\).
Then, \(m\) and \(u\) solve the continuity equation;
\begin{equation} \begin{split}
		-\Div(mDu) = -Dm\cdot Du - m \Delta u = 0.
	\end{split}\end{equation}
{ 
Generally, the function \(m\) obtained from this construction may not be nonnegative everywhere.
To construct a valid MFG solution, we define the density as the positive part,
\(\widetilde m = (m)_+\). We then define the potential \(V(x,y)\) retrospectively to satisfy the
Hamilton–Jacobi equation for the pair \((u,\widetilde m)\):
\[
V(x,y):=\widetilde m(x,y)^{1/q}-\tfrac12|D u(x,y)|^2.
\]
}
Then, we have
\[ \begin{split}
		\begin{cases}
			\frac12 |D u(x,y)|^2 +V(x,y)
			=  \widetilde m^{1/q}        & \text{ for } (x,y) \in \Omega, \\
			-\Div(\widetilde m D u) =  0 & \text{ for } (x,y) \in \Omega.
		\end{cases}
	\end{split}\]
The first equation holds in a weak-strong sense (by Definition~\ref{def:WeakSolution}), while the second equation holds in the sense of distributions.

{ 
This construction technique is particularly valuable as it provides a systematic way to generate
explicit solutions where the density vanishes (\(\widetilde m = 0\)).
These examples serve as benchmarks for understanding the structure of solutions near
empty regions. The example in the preceding section is not a particular case of this construction and shows that areas of vanishing density can arise in other cases. 
}

\section{The variational problem}\label{sec:VarProblem}

In this section, we analyze the variational formulation in Problem~\ref{problem:2} in more depth and prove the existence of minimizers. We additionally derive a few regularity estimates that arise directly from this formulation. We use the direct method in the calculus of variations, with the core difference lying in how we handle the nonstandard boundary conditions.

\subsection{On the existence of a minimizer}

In this subsection, we establish the existence of minimizers for Problem~\ref{problem:2}. As customary in the study of variational problems, we begin by establishing the following lower bound.

\begin{prop} \label{lem:lower-bound-for-I2}
	Under Assumptions~\ref{assume:data1}-\ref{assume:H} and Assumption~\ref{assume:G}, let  \( u \in \cU \) be an admissible function in Problem~\ref{problem:2}. Then, there is a constant \( C > 1 \), independent of \( u \), such that
    \[
        \mathcal{I}[u] \geq C^{-1}\|Du\|_{L^{\gamma}(\Omega)}^\gamma - C.
	\]
\end{prop}

\begin{proof} Let  \( u \in \cU \) be arbitrary. Then, from Assumption~\ref{assume:G} and Assumption~\ref{assume:H}, we obtain
	\begin{equation} \label{eq:PropOfLowerBound1}
        \begin{split}
            G( H(x, Du) ) &\geq C^{-1}\, H(x,Du)_+^{(\b+1)/\b} -C\\
            &\geq C^{-1}\, (|Du|^{\a} -C)_+^{(\b+1)/\b} -C\\
            &\geq C^{-1} |Du|^\gamma -  C,
        \end{split}
    \end{equation}
    where used that \( \gamma = \a (\b+1)/\b\).	This implies that,
    \[
        \begin{split}
    		\mathcal{I}[u] &=  \int_\Omega G( H(x, Du) )\dx - \int_{\Gamma_N} j  u \d s \\
            &\geq  C^{-1}\|Du\|_{L^{\gamma}(\Omega)}^\gamma - C- \int_{\Gamma_N} j  u \d s .
        \end{split}
	\]
    Using Lemma~\ref{lem:lower-bound-for-int-ju}, this gives us
    \[
        \begin{split}
    		\mathcal{I}[u] &\geq C^{-1}\|Du\|_{L^{\gamma}(\Omega)}^\gamma - C\,\|Du\|_{L^{\gamma}(\Omega)}-C.
        \end{split}
	\]
    Since \( \gamma>1 \), this gives us that
    \[
        \mathcal{I}[u] \geq C^{-1}\|Du\|_{L^{\gamma}(\Omega)}^\gamma - C
	\]
    for some \( C >0 \), independent of \( u \).
\end{proof}

\begin{remark}\label{rem:upperBound4cI} The previous proposition establishes that \(\inf_{u\in\mathcal{U}}  \, \mathcal{I}[u]\) is bounded below. It is also bounded above by
	\[
		\inf_{u\in\mathcal{U}} \,\mathcal{I}[u] \leq \mathcal{I}[\psi] = C.
	\]
\end{remark}
Next, we apply the direct method in the calculus of variations and prove that any minimizing sequence converges to an admissible minimizer, hence establishing the existence of a minimizer for the variational problem.

\begin{proof}[\textbf{Proof of Theorem~\ref{thm:var_prob_exist}}]
	Let \(\{ u_n \} \in \mathcal{U}\) be a minimizing sequence. By Proposition~\ref{lem:lower-bound-for-I2}, we have
	\[
        C^{-1}\|Du_n\|_{L^{\gamma}(\Omega)}^\gamma-C  \leq \mathcal{I}[u_n] \leq C.
    \]
    Consequently,
	\begin{equation} \label{eq:ProofofExistTheorem1}
        \|Du_n\|_{L^{\gamma}(\Omega)} \leq C
    \end{equation}
    for some \( C > 1 \), independent of \( n \).
    Moreover, due to Lemma~\ref{lem:lower-bound-for-int-ju} and Remark~\ref{rem:upperBound4cI}, we have that
	\begin{equation} \label{eq:ProofofExistTheorem2}
        C^{-1} \|Du_n\|_{L^{\gamma}(\Omega)}-C \leq \int_{\Gamma_N}ju_n\d s \leq C \|Du_n\|_{L^{\gamma}(\Omega)}+C.
    \end{equation}
    And the extended Poincar\'e inequality implies that
	\begin{equation} \label{eq:ProofofExistTheorem3}
        \|u_n\|_{L^{\gamma}(\Omega)} \leq C \|Du_n\|_{L^{\gamma}(\Omega)}+C \, \Big|\int_{\Gamma_N}ju_n\d s \Big|.
    \end{equation}
    Combining \eqref{eq:ProofofExistTheorem1}--\eqref{eq:ProofofExistTheorem3}, we obtain that
	\begin{equation} \label{eq:ProofofExistTheorem4}
        \|u_n\|_{L^{\gamma}(\Omega)} \leq C
    \end{equation}
    for some \( C > 1 \), independent of \( n \).

    The estimates \eqref{eq:ProofofExistTheorem1} and \eqref{eq:ProofofExistTheorem4} give us that \( \| u_n\|_{W^{1,\gamma}(\Omega)} \leq C \). Because \(W^{1,\gamma}(\Omega)\) is a reflexive Banach space, the bounded sequence \( \{u_n\} \subset W^{1,\gamma}(\Omega) \)  is weakly precompact; that is, there exists a subsequence that converges weakly to a function \(u \in W^{1,\gamma}(\Omega)\). We now show that
	\[
        u \in \mathcal{U} \quad \text{ and } \quad \liminf_{n\to\infty} \mathcal{I}[u_n] \geq \cI[u].
    \]
	To this end, note that
	\begin{align*}
		\mathcal{I}[u_n] - \mathcal{I}[u] = &
		\int_\Omega \left[G\left(  H(x, Du_n) \right) - G\left(H(x, Du) \right) \right] \dx - \int_{\Gamma_N} j \,(u_n-u) \d s \\
		&\geq \int_\Omega G'( H(x, Du) )D_pH(x,Du)(Du_n-Du)\dx - \int_{\Gamma_N} j (u_n - u) \d s.
	\end{align*}
	Because \(j \in L^{\gamma'}(\Gamma_N)\) and \(G'( H(\cdot, Du) )D_pH(\cdot,Du) \in L^{\gamma'}(\Omega)\), we have
	\[ \begin{split}
			\int_{\Gamma_N} j (u_n - u) ds \to 0,
		\end{split}\]
	and
	\[ \begin{split}
			\int_\Omega G'( H(x, Du) )D_pH(x,Du)(Du_n-Du)\dx \to 0.
		\end{split}\]

	It remains to show that \(u \in \mathcal{U}\).
	Since \(u_n \rightharpoonup u\) in \(W^{1,\gamma}(\Omega)\), we have
	\[ \begin{split}
			\lim_{n \to \infty} \int_{\Gamma_D} f (u_n - \psi) ds = \int_{\Gamma_D} f (u - \psi) ds,
		\end{split}\]
	for every \(f \in L^{\gamma'}(\Gamma_D)\).
	Taking \(f > 0\), we conclude
	\[ \begin{split}
			\int_{\Gamma_D} f (u - \psi) ds \leq 0 .
		\end{split}\]
	Therefore, \(u \leq \psi\) on \(\Gamma_D\). \end{proof}

\subsection{First-order estimates}

In this subsection, we establish several estimates that provide better regularity for solutions of the MFG. These estimates guarantee the applicability of some of our results (e.g., the Neumann trace theorem in the appendix).

\begin{lemma}
Under Assumptions~\ref{assume:data1}-\ref{assume:H} and Assumption~\ref{assume:G}, let \(u \in \cU\) be a minimizer of Problem~\ref{problem:2} and set \(m = G'(H(x,Du))\).
	Then, there is a constant \(C >0\), independent of \( m \) and \( u \), such that
	\[ \begin{split}
		m \leq C\abs{Du}^{\frac{\a}{\b}} + C.
	\end{split}\]
\end{lemma}

\begin{proof}
	Since the pair \((m,u)\) is a solution to Problem~\ref{problem:1}, we have that
	\[ \begin{split}
		m = G'(g(m)) = G'(H(x,Du)).
	\end{split}\]
	Subsequently, we have
	\[ \begin{split}
		m = G'(H(x,Du)) &\leq C|H(x,Du)|^{\frac{1}{\b}} + C\\
		&\leq C|Du|^{\frac{\a}{\b}} +C,
	\end{split}\]
	using the assumptions on \(H\) and \(G\).
\end{proof}
From this estimate, we obtain the following result.

\begin{corollary}\label{cor:m_g_m}
Under Assumptions~\ref{assume:data1}-\ref{assume:H}, and~\ref{assume:G} let \(u\) solve Problem~\ref{problem:2} and \(m = G'(H(x,Du))\).
	Then, \(m \in L^{\b+1}(\Omega)\), \(g(m) \in L^{\frac{\b+1}{\b}}(\Omega)\), and \(mD_pH(x,Du) \in \left(L^{\gamma'}(\Omega)\right)^d\).
\end{corollary}

\begin{proof}
	Using the inequality \((a+b)^{r}\leq 2^r(a^r+b^r)\), which holds for \(a,b,r \in \R_0^+\), we have
	\[ \begin{split}
		\int_\Omega m^{\b+1} \dx &\leq C \int_\Omega |Du|^{\frac{(\b+1)\a}{\b}}\dx +  C  < \infty
	\end{split}\]
	since \(\frac{(\b+1)\a}{\b} = \gamma\).

	By Assumption~\ref{assume:g}, for \(g(m)\), we have
	\[ \begin{split}
		\int_\Omega g^\frac{\b+1}{\b}  (m)\dx \leq C \int_\Omega m^{\b+1} \dx + C \leq C < \infty .
	\end{split}\]

	Now, set \(\delta = \frac{\gamma-1}{\a-1}\) and \(\delta' = \frac{\gamma-1}{\gamma- \a}\), and note that \(1/\delta + 1/\delta' = 1\). Then,
	\[ \begin{split}
		\gamma' \delta(\a-1) &= \frac{\gamma}{\gamma-1} \, \frac{\gamma-1}{\a-1}\, (\a-1) = \gamma
        \quad \text{and}
        \quad \gamma'\delta'= \frac{\gamma}{\gamma-1}\frac{\gamma-1}{\gamma- \a}= \b+1 .
	\end{split}\]
	Due to H\"{o}lder's inequality, we also have
	\[ \begin{split}
		\int_{\Omega}\abs{mD_pH(x,Du)}^{\gamma'} &\leq \left[\int_\Omega m^{\gamma'\delta'}\right]^{\frac{1}{\delta'}}\left[\int_\Omega \abs{D_pH(x,Du)}^{\gamma' \delta}\right]^{\frac{1}{\delta}}.
	\end{split}\]
	Using our assumptions on \(H\) and the inequality \((a+b)^{r}\leq 2^r(a^r+b^r)\), we further obtain
	\[ \begin{split}
		\int_{\Omega}\abs{mD_pH(x,Du)}^{\gamma'} &\leq  \left[\int_\Omega m^{\b+1}\right]^{\frac{1}{\delta'}}\left[\int_\Omega C(\abs{Du}^{\a-1}+1)^{\gamma' \delta}\right]^{\frac{1}{\delta}}\\
		&\leq  ||m||_{\b+1}^{\gamma'}\left[\int_\Omega C2^{\gamma' \delta}(\abs{Du}^{(\a-1)\gamma' \delta}+1)\right]^{\frac{1}{\delta}}\\
		&= \widetilde{C} ||m||_{\b+1}^{\gamma'} \left[||Du||_{\gamma}^{\gamma'(\a-1)}+1\right] < C,
	\end{split}\]
	since \(m\in L^{\b+1}\) and \(Du \in \left(L^{\gamma}(\Omega)\right)^d\).
\end{proof}

This leads to the following estimates.

\begin{proposition}
Let \(u\) solve Problem~\ref{problem:2} and \(m = G'(H(x,Du))\).
	Then,
	\[ \begin{split}
		\int_{\Omega} m g(m)\dx < C, \qquad \int_{\Omega} \dfrac{(m+1) |Du|^\a}{C} < C.
	\end{split}\]
\end{proposition}

\begin{proof}
	The first estimate is a direct consequence of Corollary~\ref{cor:m_g_m}.

	For the second estimate, we recall that \(m = G'(H(x,Du))\), and \(g(\cdot)\) is an inverse function of \(G'(\cdot)\), we have
	\(H(x, Du) = g(m)\), and consider an estimate
	\[ \begin{split}
		\int_\Omega (m+1)H(x,Du)\dx = \int_\Omega (m+1)g(m)\dx.
	\end{split}\]
	From Assumptions~\ref{assume:g} and~\ref{assume:H}, we obtain
	\[ \begin{split}
		\int_\Omega \dfrac{(m+1) |Du|^\a}{C}\dx \leq \int_\Omega (m+1)(g(m) +C)\dx \leq
		\int_\Omega C m g(m) + C < C.
	\end{split}\]
\end{proof}

\section{The Correspondence Result} \label{sec:Correspondance}

This section establishes the correspondence result stated in Theorem~\ref{thm:correspondance}. For clarity, the proof is divided into two major parts, each addressed in a dedicated subsection. Recall that we are working under Assumption~\ref{assume:g} and Assumption~\ref{assume:G}, which are consistent due to Lemma~\ref{lem:EquivgG}.

\subsection{From the variational formulation to the MFG}\label{subsec:DerivationOfELeq}

In this subsection, we rigorously establish that the MFG system \eqref{eq:MFG} arises as the Euler–Lagrange equation of Problem~\ref{problem:2} and that the boundary conditions \eqref{eq:BC} are naturally induced.

\paragraph{\emph{Proof of Theorem~\ref{thm:correspondance}} (\emph{I}).}

Let \(u \in \mathcal{U}\) be a minimizer of \eqref{eq:Minimization} and set
\begin{equation} \label{eq:VarFormToMFG_m=Gprim(H)}
    m(x):= G'(H(x,Du)).
\end{equation}
We now calculate the Gateaux derivative using four different perturbation functions \(v \in C^\infty(\widebar{\Omega}) \).
\begin{enumerate}
    \item \textbf{Interior equations:}
    Consider a function \(v \in C_c^\infty(\Omega)\). Define a map \(i(\ep) = \mathcal{I}[u+\ep v]\) for all \( \ep \in \R \).
    We see that \(i(0) \leq i(\ep)\), hence, \(i'(0) = 0\). In particular, due to Lemma~\ref{lem:boundOnDiffGepG} and the dominated convergence theorem, we have
    \[ \begin{split}
		0  &= \frac{d}{d \ep} \int_\Omega G\left( H(x, Du+\ep Dv) \right)\dx \Big\vert_{\ep = 0}\\
        &= \int_\Omega G'\left( H(x, Du) \right)D_pH(x,Du)\cdot Dv\dx \\
	\end{split}
    \]
    Since \( v\in C_c^\infty(\Omega) \) is arbitrary, we have $\Div (m D_pH(x,Du))=0$ in the sense of distributions.
    Accordingly,
    we obtain
    \begin{equation}\label{eq:step1PDE}
    	\begin{cases}
    		G'(H(x,Du)) = m    & \text{ for all } x \in \Omega,   \\
    		-\Div(m D_pH(x,Du)) = 0   & \text{ for all } x \in \Omega,
    	\end{cases}
    \end{equation}
    which is the PDE system governing the interior structure of the MFG.

    \item \textbf{The Neumann boundary condition:}
    Consider a function \(v \in C^\infty(\widebar{\Omega})\) such that \(v\equiv 0\) on \(\Gamma_D\). Define a map \(i(\ep) = \mathcal{I}[u+\ep v]\) for all \( \ep \in \R \).
    We see that \(i(0) \leq i(\ep)\), hence, \(i'(0) = 0\). In particular,
    \begin{equation}\label{eq:step2NeuBC1}
        \begin{split}
    		0 & = \frac{d}{d \ep}\left(\int_\Omega G\left( H(x, Du+\ep Dv) \right)\dx - \int_{\Gamma_N} j (u+\ep v) \d s\right)\Big\vert_{\ep = 0}\\
            &= \int_\Omega G'\left( H(x, Du) \right)D_pH(x,Du)\cdot Dv\dx - \int_{\Gamma_N} j v \d s\\
    		&= \int_\Omega m \, D_pH(x,Du)\cdot Dv\dx - \int_{\Gamma_N} j v \d s.
        \end{split}
    \end{equation}
  { 
  The previous steps ensure that
\[
\Div\bigl(m D_pH(x,Du)\bigr)=0
\quad \text{in the sense of distributions in } \Omega.
\]
Moreover, since \(m D_pH(x,Du) \in L^{\gamma'}(\Omega;\mathbb{R}^d)\), it follows that
\[
m D_pH(x,Du) \in W^{\gamma'}(\Div;\Omega),
\]
so that its normal trace on \(\Gamma\) is well defined by Theorem~\ref{thm:mainNormalTrace}.
  }
   Also, because $u\in W^{1, \gamma}$ and $m\in L^{\beta+1}$, using Assumption \ref{assume:H}, we get $m D_pH(x, Du)\in L^{\gamma'}$. Thus, $m D_pH(x, Du)\in W^{\gamma'}( \Div; \Omega)$.
    Accordingly, we can invoke Theorem~\ref{thm:mainNormalTrace}, and establish that
    \[
        \int_{\Gamma_N} j \, v \d s = \int_\Omega m \, D_pH(x,Du)\cdot Dv\dx = \int_{\Gamma_N}  v\, mD_pH(x,Du) \cdot \nu  \d s.
    \]

    Since \(v\) is arbitrary, we obtain
    \begin{equation}\label{eq:step2NeuBC2}
    		m D_pH(x,Du)\cdot \nu = j \qquad \text{ for all } x \in \Gamma_N.
    \end{equation}

    \item \textbf{The no-entry condition:}
    Consider a function  \(v \in C^\infty(\widebar{\Omega})\) such that \(v(x) \leq 0\)  for all \(x \in \Gamma_D\). Define a map \(i(\ep) = \mathcal{I}[u+\ep v]\) for all \( \ep \geq 0 \). Note that  \( u+\ep v \in \mathcal{U} \) for all \( \ep \geq 0 \); hence, \( i(\ep) \geq i(0) \). Using \eqref{eq:step1PDE} and \eqref{eq:step2NeuBC2} above, we subsequently have
    \[
        \begin{split}
            0 \leq i'(0) &=  \int_\Omega m\, D_pH(x, Du) \cdot Dv \dx - \int_{\Gamma_N} j \, v  \d s \\
    	  & =    \int_{\Gamma} v \, m D_pH(x, Du) \cdot \nu \d s-\int_{\Gamma_N}  j \, v  \d s =\int_{\Gamma_D} v \, m D_pH(x, Du) \cdot \nu \d s.
        \end{split}
    \]
    Since \(v|_{\Gamma} \leq 0\) is arbitrary, we obtain
    \begin{equation} \begin{split} \label{eq:step2BC}
		m D_pH(x, Du) \cdot \nu  \leq 0 \qquad \text{on } \Gamma_D.
	\end{split}\end{equation}

    \item \textbf{The contact-set condition:}
    Consider an arbitrary function  \(v \in C^\infty(\widebar{\Omega})\). Define a map \(i(\ep):= \mathcal{I}[u+\ep (\psi-u)v]\) for all \( \ep \in (-1,1) \) small enough so that \(\sup_{\Gamma_D} |\ep v| < 1\). Note that,
    \[
        v (\psi-u) \in W^{1,\gamma}(\Omega) \quad \text{ and } \quad u+\ep v (\psi-u) \leq \psi \enskip \text{ on } \Gamma_D;
    \]
    hence, \(u+ \ep v (\psi-u) \in \mathcal{U}\). Subsequently, \(i'(0) = 0\) and, similar to before,
    \[
        0  = 	\int_{\Gamma_D}  v (\psi-u) \, m D_pH(x, Du) \cdot \nu \d s.
	\]
    Since \( v \) is arbitrary, we obtain
    \begin{equation} \begin{split} \label{eq:step3BC}
		(\psi-u) m D_pH(x, Du) \cdot \nu = 0 \qquad \text{on } \Gamma_D.
	\end{split}\end{equation}
\end{enumerate}
Combining \eqref{eq:step1PDE} with \eqref{eq:step2NeuBC2}--\eqref{eq:step3BC} gives us the MFG \eqref{eq:MFG}-\eqref{eq:BC}, implying that the MFG System in Problem~\ref{problem:1} is the Euler-Lagrange equation for the variational formulation in Problem~\ref{problem:2}. Consequently, any minimizer of Problem~\ref{problem:2} is a solution to Problem~\ref{problem:1}. \hfill $\square$.

\subsection{From the MFG to the variational formulation}\label{subsec:MFGtoVarForm}

What we have established above in Subsection~\ref{subsec:DerivationOfELeq} is the first part of Theorem~\ref{thm:correspondance}. Namely, we proved that a minimizer for Problem~\ref{problem:2} produces a weak solution to Problem~\ref{problem:1}. In this subsection, we prove the converse, which states that a weak solution to Problem~\ref{problem:1} produces a minimizer for Problem~\ref{problem:2}.

\begin{proof}[Proof of Theorem~\ref{thm:correspondance}(II)]
Let \((m,u)\) be a weak solution to Problem~\ref{problem:1} in the sense of Definition~\ref{def:WeakSolution}.
We will show that for all \(w \in \mathcal{U}\),
	\[ \begin{split}
			\mathcal{I} [u] \leq  \mathcal{I} [w].
		\end{split}\]
	Indeed, due to the convexity of \(G\) and \(H\), we have
	\begin{align*}
		\mathcal{I} [w]-  \mathcal{I} [u]
		=    & \int_\Omega G(H(x, Dw)) - G(H(x, Du)) \dx - \int_{\Gamma_N} j (w - u)  \d s       \\
		\geq & \int_\Omega G'(H(x, Du))  D_pH(x, Du) D(w - u) \dx - \int_{\Gamma_N} j (w - u)  \d s\\
        = & \int_{\Omega} m  D_pH(x, Du) D(w - u) \dx - \int_{\Gamma_N} j (w - u)  \d s,
	\end{align*}
    where we used that \( G'(H(x, Du)) = G'(g(m)) = m\). Now, by our definition of weak solutions (Definition~\ref{def:WeakSolution}) and the normal trace theorem (Theorem~\ref{thm:mainNormalTrace}), we obtain
	\[
        \begin{split}
        \mathcal{I} [w] -  \mathcal{I} [u] \geq & \int_{\Omega} m  D_pH(x, Du) D(w - u) \dx - \int_{\Gamma_N} j (w - u)  \d s \\
    	=& \int_{\Gamma} (w - u)\, m  D_pH(x, Du)\cdot\nu \d s- \int_{\Gamma_N} j (w - u)  \d s\\
    	=& \int_{\Gamma_D} (w - u)\, m  D_pH(x, Du)\cdot\nu \d s.
        \end{split}
    \]

    Finally, due to our assumptions, \( (w-\psi) \, m  D_pH(x, Du)\cdot\nu \geq 0 \) and \( (\psi - u)\, m  D_pH(x, Du)\cdot\nu  = 0 \)  on \( \Gamma_D \), giving us that
    \[
        \mathcal{I} [w] -  \mathcal{I} [u] \geq  \int_{\Gamma_D} [(w - \psi)+(\psi-u)]\, m  D_pH(x, Du)\cdot\nu \d s \geq 0.
    \]
    This concludes the proof.  
    \end{proof}

\section{Monotonicity and the uniqueness of solution}\label{sec:monoton_uniquenessMFG}
In this section, we establish uniqueness results for the MFG system by employing the monotonicity techniques developed in \cite{lasryMeanFieldGames2007}. We begin with a discussion on the related monotone MFG operator, which was introduced in \cite{FGT1}. Subsequently, we demonstrate the uniqueness of the density function \( m \), while the uniqueness of the gradient \( Du \) is guaranteed in the region where the density remains strictly positive, i.e., where \( m > 0 \).

\subsection{The MFG operator}

Set \( \cX:= L^{\b+1}(\Omega)\times W^{1,\gamma}(\Omega) \) and let \(\cX'\) be its dual. Let \( \cY \subset \cX \) be the family of all functions \( (m,u) \in \cX \) that satisfy the conditions:
\begin{enumerate}[label=\roman*.]
        \item The density $m$ is nonnegative.
        \item The Dirichlet boundary condition in \eqref{eq:BC} is satisfied in the trace sense.
        \item The divergence of the flow, \( -\Div( mD_pH(x,Du)) \), is in \(L^{\gamma'}(\Omega) \).
        \item The Neumann boundary conditions in \eqref{eq:BC} are satisfied in the \emph{Neumann trace} sense.
\end{enumerate}
The \emph{MFG operator} \( A : \cY \to \cX' \) we investigate in this section is given by
\begin{equation} \label{eq:MFGop}
		A \begin{bmatrix} m \\[.1em] u \end{bmatrix}  =  \begin{bmatrix} -H(x,Du)+g(m) \\[.25em] -\Div(mD_pH(x,Du))\end{bmatrix},
\end{equation}
which was introduced in \cite{FGT1} inspired by the uniqueness proof in \cite{lasryMeanFieldGames2007}. The action of this operator is given by the pairing
\begin{equation}\label{eq:A_prod}
	\begin{split}
		\Big\langle A \begin{bmatrix} m \\[.1em] u \end{bmatrix}  \, ,\,  \begin{bmatrix}  \mu \\[.1em]  v  \end{bmatrix}  \Big\rangle = \int_\Omega (-H(x,Du)+g(m)) \, \mu\dx &-  \int_\Omega \Div(mD_pH(x,Du)) \,  v  \dx\\
        = \int_\Omega (-H(x,Du)+g(m)) \,  \mu\dx &+ \int_\Omega mD_pH(x,Du)\cdot D v  \dx \\
        &\phantom{44444} -\int_{\Gamma} v \,\, mD_pH(x,Du)\cdot \nu \d s.
\end{split}
\end{equation}

This subsection's main objective is to establish the MFG operator's monotonicity. This property, which is defined as follows, allows us to prove the MFG's uniqueness results.

\begin{definition}  Let \( \cX \)  be a Banach space and \( \cX' \) its dual. An operator \( A : \cY \subset \cX \to \cX' \) is \emph{monotone} if and only if, for all \(a,b \in \cY\), it satisfies
	\[ \begin{split}
		\left \langle A(a)-A(b)\,\, ,\,\, a-b \right \rangle \geq 0.
	\end{split}\]
\end{definition}

\begin{remark}
    The well-definedness of the pairing in \eqref{eq:A_prod} relies on the integrability of the constituent terms. While Assumptions~\ref{assume:data1}–\ref{assume:g} ensure that
    \[
        -H(x, Du) + g(m) \in L^{(\beta+1)/\beta}(\Omega)
        \quad \text{and} \quad
        m D_p H(x, Du) \in L^{\gamma'}(\Omega; \mathbb{R}^d)
    \]
    for all pairs \( (m,u) \in L^{\b+1}(\Omega)\times W^{1,\gamma}(\Omega) \), imposing the additional constraint
    \[
        \Div(m D_p H(x, Du)) \in L^{\gamma'}(\Omega)
    \]
    allows us to apply the normal trace theorem (Theorem~\ref{thm:mainNormalTrace}) and guarantees that
    \[
        m D_p H(x, Du) \cdot \nu \in W^{-1 + 1/\gamma', \gamma'}(\Gamma).
    \]
\end{remark}

\subsection{The monotonicity of the MFG operator}
The following result establishes the monotonicity of the mean-field game operator.
\begin{prop} \label{prop:MonotonicityOfA}
	The operator \(A: \cY \to   \cX'\), as defined in \eqref{eq:MFGop}, is monotone on \( \cY \).
\end{prop}

\begin{proof}
    Let \((m,u)\) and \((\mu, v )\) be two pairs in \( \cY \). We have
    \begin{equation}
    \begin{split} \label{eq:pfOfMono1}
        \Big\langle A \begin{bmatrix} m \\[.1em] u \end{bmatrix}  - A \begin{bmatrix} \mu \\[.1em]  v  \end{bmatrix}  \, &,\,  \begin{bmatrix}  m-\mu \\[.1em] u- v  \end{bmatrix} \Big\rangle \\
        = & \int_\Omega\left[g(m)-g(\mu)\right] (m-\mu)\dx\\
        &\quad + \int_\Omega\left[-H(x,Du)+H(x,D v ) \right] (m-\mu)\dx\\
        &\quad+\int_\Omega \left[mD_pH(x,Du) -\mu D_pH(x,D v )\right] (Du-D v ) \dx\\
        &\quad +\int_{\Gamma} \left[-mD_pH(x,Du) +\mu D_pH(x,D v )\right]\cdot\nu (u- v ) \ds.
    \end{split}
    \end{equation}
    Now, we notice that since \((m,u)\) and \((\mu, v )\) satisfy the boundary conditions \eqref{eq:BC}, we have, on the one hand,
    \begin{equation} \label{eq:pfOfMono2}
        mD_pH(x,Du)\cdot \nu = \mu D_pH(x,D v )\cdot \nu = j \qquad \text{on }\Gamma_N.
    \end{equation}
	On the other hand, on \(\Gamma_D\), we have
	\begin{equation} \begin{split}
			\label{eq:pfOfMono3}
			-(u-\psi) m D_pH(x, Du) \cdot \nu &= 0,\qquad
			(\psi- v ) \mu D_pH(x, D v ) \cdot \nu = 0,\\
			-(\psi- v ) m D_pH(x, Du) \cdot \nu &\geq 0,  \qquad	(u-\psi) \mu D_pH(x, D v ) \cdot \nu \geq 0.
	\end{split}\end{equation}
	Combining \eqref{eq:pfOfMono2} and \eqref{eq:pfOfMono3}, we obtain that
	\begin{equation} \begin{split} \label{eq:pfOfMono4}
			I_1:= \int_{\Gamma} (u- v ) \left[-mD_pH(x,Du) +\mu D_pH(x,D v )\right]\cdot\nu \d s \geq 0.
	\end{split}\end{equation}
	Moreover, because \(H(x,p)\) is convex in \(p\), we have
	\[ \begin{split}
		H(x,Du)-H(x,D v ) &\geq (Du-D v )\cdot D_pH(x,D v ),
	\end{split}\]
	and
	\[ \begin{split}
		H(x,D v )-H(x,Du) &\geq (D v -Du)\cdot D_pH(x,Du).
	\end{split}\]
	Multiplying the first inequality by \(\mu\) and the second by \(m\) and summing the two equations, we obtain
	\begin{equation} \label{eq:pfOfMono5}
            \begin{split}
                I_2:= \int_\Omega (m-\mu)&\left( H(x,D v )-H(x,Du)\right) \\
                + &\int_\Omega \left(mD_pH(x,Du)-\mu D_pH(x,D v )\right) \cdot (Du-D v ) \dx \geq 0.
    	\end{split}
        \end{equation}
	Finally, since \(g\) is monotone, we have that
	\begin{equation} \label{eq:pfOfMono6}
			I_3:= \int_\Omega(g(m)-g(\mu)) (m-\mu)\dx \geq 0.
    \end{equation}

	Due to inequalities \eqref{eq:pfOfMono4}, \eqref{eq:pfOfMono5}, and \eqref{eq:pfOfMono6}, inequality \eqref{eq:pfOfMono1} becomes
	\[ \begin{split}
		\Big\langle A \begin{bmatrix} m \\[.1em] u \end{bmatrix}  - A \begin{bmatrix} \mu \\[.1em]  v  \end{bmatrix}  \, ,\,  \begin{bmatrix}  m-\mu \\[.1em] u- v  \end{bmatrix} \Big\rangle = I_1+I_2+I_3 \geq 0.
	\end{split}\]
	This concludes the proof.
\end{proof}

\subsection{On the Uniqueness of Solution to the MFG}
\label{73}

Here, we use the previously proven monotonicity to establish that $Du$ is unique in the region where the density function $m$ is strictly positive, implying that the density function $m$ itself is also unique in $\Omega$.

\paragraph{\textit{\textbf{Proof of Theorem~\ref{thm:uniquenessMFG}.}}} The proof utilizes the Lasry-Lions monotonicity method \cite{ll2}.

    Using the notation in the proof of Proposition~\ref{prop:MonotonicityOfA}, we see that
	\[ \begin{split}
		A\begin{bmatrix} m \\[.1em] u \end{bmatrix}= A \begin{bmatrix} \mu \\[.1em]  v  \end{bmatrix} =  \begin{bmatrix} 0 \\ 0 \end{bmatrix} .
	\end{split}\]
	Therefore, due to inequalities \eqref{eq:pfOfMono4}, \eqref{eq:pfOfMono5}, and \eqref{eq:pfOfMono6}, we have \(I_1 = I_2 = I_3 = 0\). By the inequalities in \eqref{eq:pfOfMono3}, we see that \(I_1 = 0\) implies that
	\[ \begin{split}
		(\psi-u)\eta D_pH(x,D v )\cdot \nu = 0  \quad \text{and} \quad -(\psi- v )m D_pH(x,Du)\cdot \nu = 0.
	\end{split}\]
	This establishes 1. Now, since \(g\) is strictly monotone, we have that
	\[ \begin{split}
		(g(m)-g(\eta))(m-\eta) > 0
	\end{split}\]
    whenever \(\eta \neq m\).	However, because \(I_3 = 0\), we have that \((g(m)-g(\eta))(m-\eta) = 0\) almost everywhere in \(\Omega\). Hence, \(\eta = m\) almost everywhere in \(\Omega\). This establishes 2. For 3, we use the strict convexity assumption on \(H(x,p)\) in the variable \(p\). Indeed, given that \(I_2 =0\), we see that whenever \(m=\eta \neq 0\), then
	\[ \begin{split}
		H(x,Du)-H(x,D v ) = (Du-D v )\cdot D_pH(x,D v ).
	\end{split}\]
	However, this is only possible when \(Du=D v \). This ends the proof. \hfill $\square$

\subsection{On the uniqueness of minimizer}\label{sec:uniquness}

Given compatible coupling term \( g \) and mapping \( G \), the inherent correspondence established in Theorem~\ref{thm:correspondance} between Problem~\ref{problem:1} and Problem~\ref{problem:2} naturally induces a correspondence between results in the MFG setting and the variational formulation setting. In the variational formulation setting, this correspondence gives rise to the following two results that serve as the counterpart to the uniqueness result of Theorem~\ref{thm:uniquenessMFG}.

\begin{corollary}\label{cor:UniquenessInVarForm}
	Under Assumptions~\ref{assume:data1}--\ref{assume:H} and Assumption~\ref{assume:G}, let \(u,\,v \in \mathcal{U}\) be two minimizers of Problem~\ref{problem:2}. The following must hold:
	\begin{enumerate}
		\item For almost all \(x\in \Gamma_D\), we have
		\[
            (\psi-u)G'(H(x,D v )) D_pH(x,D v )\cdot \nu = (\psi- v )G'(H(x,Du)) D_pH(x,Du)\cdot \nu = 0.
		\]
		\item For almost all \(x \in \Omega\),
		\[
		  G'\bigl(H(x,Du(x))\bigr) = G'\bigl(H(x,Dv(x))\bigr).
		\]
		\item For almost all \(x \in \Omega\),
		\[
		  G'\bigl(H(x,Du(x))\bigr) \neq 0 \implies
		  Du(x) = Dv(x).
		\]
	\end{enumerate}
\end{corollary}

\begin{proof}
	The claim follows directly from Theorem~\ref{thm:correspondance} and Theorem~\ref{thm:uniquenessMFG}.
\end{proof}

As a subsequence of Corollary~\ref{cor:UniquenessInVarForm} above, we obtain the following corollary related to boundary behavior.
\begin{corollary}
	Under Assumptions~\ref{assume:data1}-\ref{assume:H} and Assumption~\ref{assume:G}, let \(u,\, v \in \mathcal{U}\) be two minimizers of Problem~\ref{problem:2}.	The following must hold:
	\begin{enumerate}
        \item For almost all \(x \in \Omega\),
		\[
		  G\bigl(H(x,Du(x))\bigr) = G\bigl(H(x,Dv(x))\bigr).
		\]
		\item
		\[
            \int_{\Gamma_N}j \, v\d s = \int_{\Gamma_N}j\, u \d s.
        \]
    \end{enumerate}
\end{corollary}
\begin{proof}

	\begin{enumerate}[leftmargin=1.5em]
	    \item Due Corollary~\ref{cor:UniquenessInVarForm}.3, we immediately obtain that
            \[
                G'(H(x,Du)) \neq 0 \implies G(H(x,Du))=G(H(x,Dv)).
            \]
            On the other hand, let
            \[
                G_0 = \inf_{z \in \R} G(z).
            \]
            Since \( G \) is increasing, whenever \( G'(z) = 0\), we have that \( G(z) = G_0\).
            { 
            On the other hand, by Corollary~\ref{cor:UniquenessInVarForm}.2 we have
\(G'(H(x,Du)) = G'(H(x,Dv))\).
Let \(G_0 = \inf_{z\in\mathbb{R}} G(z)\).
Since \(G\) is convex and increasing (on its support), whenever \(G'(z)=0\)
we must have \(G(z)=G_0\).
Therefore, if \(G'(H(x,Du))=0\), we have \(G'(H(x,Dv))=0\),
which implies \(G(H(x,Du))=G_0\) and \(G(H(x,Dv))=G_0\).
This establishes the first claim.
            }

        \item The second claim follows from the first. Because \( \mathcal{I} [u] =  \mathcal{I} [v]\), we have
        \[
        \begin{split}
            \mathcal{I} [u] &=  \int_\Omega G\left( H(x, Du) \right)\dx - \int_{\Gamma_N} j u \d s\\
            &=\int_\Omega G\left( H(x, Dv) \right)\dx - \int_{\Gamma_N} j v \d s + \int_{\Gamma_N} j v - \int_{\Gamma_N} j u \\
            &= \mathcal{I} [v]+ \int_{\Gamma_N} j v - \int_{\Gamma_N} j u.
		\end{split}\]
        Therefore,
        \[
            \int_{\Gamma_N} j v = \int_{\Gamma_N} j u.
        \]
    \end{enumerate}
    This concludes the proof.
\end{proof}

\appendix
\renewcommand{\theequation}{a.\arabic{equation}}

\section*{Appendix}
\section{The Neumann Trace and Neumann Boundary Condition}\label{appendix_A}

In this appendix, we address the weak formulation of the Neumann boundary condition for our MFG system under low-regularity assumptions, using the \emph{normal (Neumann) trace} theorem. A rigorous proof of the well-definedness of this operator is provided in \cite{Galdi2011} for the Banach space \(\tilde{H}_p(\Omega)\), corresponding to \(W^{p}(\Div;\Omega)\) in our notation.  A more accessible outline of the theorem in the Hilbert-space setting \(H(\Div;\Omega)\) appears in \cite{Tartar1}.
For completeness, we present the proof of this result, drawing on the techniques of \cite{Tartar1}, \cite{Galdi2011}, \cite{DiBenedetto1}, and \cite{evansPartialDifferentialEquations2010}.

{ 
 To keep notation consistent with the references used, throughout this appendix, we denote the Lebesgue integrability exponent by $p$ and its conjugate by $q$. In the context of the main results of this paper, these correspond to the exponents $\gamma'$ and $\gamma$, respectively. This is distinct from the momentum variable $p$ used in the Hamiltonian $H(x,p)$.
We refer to the normal trace as the \emph{Neumann trace} to avoid confusion with the standard boundary trace operator.
}

We note that, to maintain the coherence of this appendix, we use the standard exponent \( p \) for our \( L^p \)--spaces and denote its convex conjugate by  \( q \). So, in what follows, \( \frac{1}{p}+\frac{1}{q} = 1\). For the main body of this paper, { we substitute $p = \gamma$ and $q = \gamma'$.} We use the term \emph{Neumann trace} to avoid ambiguity when discussing the \emph{usual} boundary trace in the same context.

\subsection{The weak divergence}
Let \(\Omega \subset \R^d\) be an open, bounded domain with smooth boundary $\Gamma = \del \Omega$. Given sufficiently regular vector field \( \mathbf{u} \) and function \( \phi \), the divergence theorem can be expressed as the integration-by-parts identity
\begin{equation}  \label{eq:DivThm}
		\int_\Omega \phi \, \Div (\mathbf{u}) \dx = - \int_\Omega \mathbf{u} \cdot D\phi \dx + \int_{\Gamma} \phi \enskip \mathbf{u} \cdot \nu  \d s.
\end{equation}
This identity constitutes the foundational structure for the weak formulation of Neumann boundary conditions in divergence-form PDE.
Our goal is to provide sufficient conditions for the Neumann trace  \( \mathbf{u} \cdot \nu \) to be well-defined and to extend this identity to a larger class of vector fields whose divergence may only be weakly defined in the following sense.

\begin{definition} \label{def:weakDiv} The \emph{weak divergence} of a vector-field \( \mathbf{u} \in L^1(\Omega; \R^d) \) is the \emph{linear functional} \( \Div (\mathbf{u})  \in  L^1_{loc}(\Omega) \) that satisfies
\begin{equation}
	\label{wd}
    \langle \Div (\mathbf{u}) \, , \phi\rangle := \int_\Omega \Div (\mathbf{u})  \phi = - \int_\Omega  \mathbf{u}\cdot D\phi \dx
\end{equation}
	for every \(\phi \in C_c^\infty(\Omega)\).
\end{definition}
To extend this definition to \eqref{eq:DivThm}, which incorporates boundary traces, we need to use a class of test functions larger than \( \mathcal{D}(\Omega) = C^{\infty}_c(\Omega) \). In particular, we use the following family of test functions.
\begin{definition} \label{Def:TestFuns}
	We define the family of \emph{test functions} \( \mathfrak{D}(\widebar{\Omega}) \) by
    \[
        \{ \,  \phi|_{\widebar{\Omega}} \, \colon \,  \phi \in C^\infty(\R^d) \, \},
    \]
    which is the set of restrictions of \(C^\infty\)--functions that are defined on the entire Euclidean space \( \R^d \).
\end{definition}
\begin{remark}
    We denote this family of test functions \( \mathfrak{D}(\widebar{\Omega}) \) to avoid ambiguity with commonly used family of test functions \( \mathcal{D}(\Omega) = C^{\infty}_c(\Omega) \).
\end{remark}

\subsection{Divergence Sobolev space}

For the Neumann trace to be well-defined in the framework of weak solutions, we need to choose an appropriate function space and a domain \( \Omega \) whose boundary enjoys an appropriate level of regularity. The space we are considering in this section is the following \emph{divergence Sobolev space}.
\begin{definition} \label{def:DivSobolevSpace}
	Let \(\Omega \subset \R^d\) be an open, bounded domain, and suppose that $1 \leq p < \infty$.
	The \emph{divergence Sobolev space} \( W^{p}( \Div; \Omega) \) 	is the space of vector fields
	\[
    \begin{split}
			W^{p}(\Div; \Omega):= \{ \mathbf{w} \in L^p(\Omega; \R^d)\, : \, \Div(\mathbf{w}) \in L^p(\Omega)\},
    \end{split}
    \]
	where \( \Div(\mathbf{w}) \) is defined weakly according to Definition~\ref{def:weakDiv}. This space is equipped with the natural norm
	\[ \begin{split}
			\| \mathbf{w} \|_{W^{p}( \Div; \Omega)}  := \Big( \|\mathbf{w}\|_{L^p(\Omega; \R^d)}^p + \|\Div(\mathbf{w})\|_{L^p(\Omega)}^p\Big)^{1/p}.
		\end{split}\]
\end{definition}
\begin{remark}
    We use this name \emph{divergence Sobolev space} consistently with common \emph{divergence-measure fields} known in the literature \cite{ChenFrid2001}.
\end{remark}
The regularity condition we require the domain \(\Omega\) to have on its boundary is the following variant of the \emph{segment condition} (see \cite{evansPartialDifferentialEquations2010}, \cite{DiBenedetto1}).
\begin{definition}
A bounded domain \(\Omega \subset \R^d\) satisfies the \emph{segment condition} on its boundary \( \Gamma = \del \Omega\) if there are two finite collections
\[
\{ x_n \}_{n=1}^N \subset \Gamma\quad  \text{and} \quad  \{ y_n \}_{n=1}^N \subset \R^d \backslash \{0\}
\]
and two positive constants \( r \) and \( t^* \) such that
	\begin{enumerate}[label = \roman*.]
		\item the collection \( \mathcal{B}_N = \{ B_{r}(x_{n}) \}_{n=1}^N \)  of open balls covers \( \Gamma \), and
		\item for \( n=1,..., N \), we have
		      \[ \begin{split}
				      B_{t}(x+ty_n)  \subset \Omega
			      \end{split}\]
            for all \( x \in B_{r}(x_n) \cap \widebar{\Omega} \) and \(t \in (0,t^*)\).
	\end{enumerate}
\end{definition}

\begin{remark}
	This condition holds when \(\Gamma\) is \(C^1\) \cite{evansPartialDifferentialEquations2010}.
\end{remark}

\begin{lemma}
	Let \( \Omega \subset \R^d \) be an open, bounded domain. The space \(W^{p}( \Div; \Omega)\) is a Banach space.
\end{lemma}
\begin{proof}
	We only prove the completeness property. The remaining normed--space properties are trivial.
    \begin{enumerate}
        \item Consider an arbitrary Cauchy sequence in \(\left(\mathbf{u}_n\right)_{n\geq 1}\) in  \(W^{p}( \Div; \Omega)\).
	Then, the sequence \(\left(\mathbf{u}_n\right)_{n\geq 1}\) and \(\left(\Div(\mathbf{u}_n)\right)_{n\geq 1}\) are Cauchy sequences in \(  L^p(\Omega; \R^d) \) and \(L^p(\Omega)\), respectively.
	By the completeness of \( L^p \)--spaces, there is a vector field \(\mathbf{u} \in  L^p(\Omega; \R^d)\) and a function \( v \in L^p(\Omega)\) such that
    \[
        \mathbf{u}_m \to \mathbf{u}  \quad \text{and} \quad  \Div(\mathbf{u}_m) \to v \in L^p(\Omega).
    \]

    \item Note that
	\[
        \begin{split}
			\int_\Omega \Div (\mathbf{u})  \phi = - \int_\Omega \mathbf{u}\cdot D\phi &= \lim_{n\to \infty} - \int_\Omega \mathbf{u}_n\cdot D\phi \\
            &= \lim_{n \to \infty} \int_\Omega \Div (\mathbf{u}_n) \phi
			= \int_\Omega v \phi.
		\end{split}
    \]
        for all \(\phi \in C^\infty_c (\Omega)\). Due to the density of \(\phi \in C^\infty_c (\Omega)\) in \( L^{p'} (\Omega) \), we have that
    \[
		\|\Div (\mathbf{u}) - v\|_{L^{p} (\Omega)} = \sup_{\substack{\phi \in C^\infty_c (\Omega) \\ \|\phi \|\leq1}}\int_\Omega (\Div (\mathbf{u})-v)  \phi = 0.
    \]
    This shows that \( \Div (\mathbf{u}) = v \) almost everywhere in \( \Omega \).
    \end{enumerate}
    This concludes the proof.
\end{proof}

Our goal in this section is to define a \emph{trace operator} on \(W^{p}( \Div; \Omega)\) in a way that is consistent with the usual trace definition.
We start by establishing the following result.

\begin{prop} \label{prop: density}
	Let \(\Omega\) be a bounded domain in \(\R^d\) that satisfies the segment condition described above.
	Then, \(\left(\mathfrak{D} (\widebar{\Omega})\right)^d\) is dense in \(W^{p}( \Div; \Omega)\).
\end{prop}

\begin{proof} Let \(\mathbf{u} = (u^1,...,u^d) \in W^{p}( \Div; \Omega)\).
\begin{enumerate}
    \item Let \(B_r({x_0})\in \mathcal{B}_N\), where \(\mathcal{B}_N\) is the cover from the segment condition, and let \(y_0\) be the vector for which \(B_{t}(x+ty_0) \subset \Omega\) for all \(t \in (0,t^*)\) and \(x \in \widebar{\Omega} \cap B_r(x_0)\).

	\item  For every \(t\in (0,t^*)\) and \(x \in \widebar{\Omega}\cap B_r(x_0)\),  define \(\mathbf{u}_t\) to be
    \[
        \mathbf{u}_{t}(x) = (\mathbf{u}^1_{t},...,\mathbf{u}^d_{t}) := \mathbf{u}(x+ty_0),
    \]
    and define \(\mathbf{v}_t\) to be its mollified version
    \[
        \mathbf{v}_t := \eta_t\ast \mathbf{u}_t(x),
    \]
    where \(\eta_t(x) = c_t\eta(x/t)\) is the standard mollifier.

	\item   First, note that
            \begin{equation} \label{eq:DensityThm1}
            \|\Div \mathbf{v}_t - \Div \mathbf{u}\|_{L^p(B_r(x_0))} \leq \|\Div \mathbf{v}_t - \Div \mathbf{u}_t\|_{L^p(B_r(x_0))} +\|\Div \mathbf{u}_t - \Div \mathbf{u}\|_{L^p(B_r(x_0))} .
            \end{equation}
            On the right-hand side, the second term vanishes in the limit due to the continuity of translations in \(L^p\) spaces.
            The first term vanishes in the limit
            following the technique used for the classical weak derivative (see \cite{evansPartialDifferentialEquations2010}). Thus, as \( t \to 0\), \( \mathbf{v}_t \to \mathbf{u}\) strongly in \( W^{p}(\Div; \Omega)\).
    \item Let \( \Omega_0 \subset\subset  \Omega\) be an open subset, compactly contained in \( \Omega \), such that
    \[
       \Omega \subset \Omega_0 \cup \; \bigcup_{n=1}^N B_{r}(x_n).
    \]
    Denote by \( \mathbf{v}^n_t \) the smooth approximation of \( \mathbf{u} \) in \( B_{r}(x_n) \). Also, denote by \( \mathbf{v}^0_t \) the smooth approximation of \( \mathbf{u} \) in \( \Omega_0 \), which exists following the above procedure without the need for translation. Let \( \{ \xi_n \}_{n=0}^N\) be a smooth partition of unity corresponding to the open cover above, and define
    \[
        \mathbf{v}_t = \sum_{n=0}^N \xi_n\mathbf{v}_t^n.
    \]
    Since the summation is finite, and \( \sum_{n=0}^N \xi_n =1 \), we have that
    \[
    \begin{split}
        \|\mathbf{v}_t -\mathbf{u}\|_{W^{p}( \Div; \Omega)}  &\leq \sum_{n=0}^N (\|\Div (\xi_n \mathbf{v}^n_t) - \Div(\xi_n \mathbf{u})\|_{L^p(\Omega)} + \|\xi_n \mathbf{v}^n_t - \xi_n \mathbf{u}\|_{L^p(\Omega; \R^d)})\\
    &\leq C \|\mathbf{v}^0_t -\mathbf{u}\|_{W^{p}( \Div; \Omega_0)} + C \sum_{n=1}^N \| \mathbf{v}^n_t - \mathbf{u}\|_{W^{p}( \Div; B_{r}(x_n))},
    \end{split}
    \]
    which goes to 0 as \( t \to 0\). \qedhere
\end{enumerate}
\end{proof}
\begin{remark}
    We omitted the part on the convergence of \(\|\xi_n \mathbf{v}^n_t - \xi_n \mathbf{u}\|_{L^p(\Omega; \R^d)}\) since it is standard (see \cite{evansPartialDifferentialEquations2010} for example).
\end{remark}

\subsection{The Neumann Trace}\label{sec:normalTrace}

Let \(\tr_0\) be the usual (Dirichlet) trace operator on the space \(W^{1,q}(\Omega)\). We observe that, given a smooth vector-field \(\mathbf{u} \in \left(\mathfrak{D}(\widebar{\Omega})\right)^d\) and a function \(v \in W^{1,q}(\Omega)\), we have
\begin{equation} \begin{split} \label{eq: Trace on Smooth}
		\int_\Omega \left[\mathbf{u} \cdot Dv + v\, \Div \mathbf{u}  \right]\,\dx = \int_{\Gamma} \mathbf{u} \cdot \nu \enskip \tr_0 v \,\d s,
	\end{split}\end{equation}
where \(\nu\) is the exterior normal to \( \Gamma \). We aim to use the density result above to extend the Neumann trace \(\mathbf{u} \cdot \nu \) to vector-fields \( \mathbf{u}  \in W^{p}( \Div; \Omega) \) using this identity. In particular, we want to define \( \mathbf{u} \cdot \nu := \tr_1(\mathbf{u} )\) as the \( W^{p}( \Div; \Omega) \)--completion of a bounded linear map \( \tr_1 : \left(\mathfrak{D}(\widebar{\Omega})\right)^d \to  W^{-1+\frac{1}{q}, p}(\Gamma)\), where
\[
    W^{-1+\frac{1}{q}, p}(\Gamma) := \left( W^{1-\frac{1}{q}, q}(\Gamma) \right)'
\]
is the dual of the \emph{trace space} of \( W^{1, q}(\Omega) \).

We establish the validity of this definition via the extension of functions \( v \in W^{1-\frac{1}{q}, q}(\Gamma)\), defined on the boundary, to the interior of the domain \( \Omega \) via an appropriate lifting operator, which is provided by the following lemma.
\begin{lemma} \label{lemma:liftingOp}
	Let \(\Omega\) be a bounded, open, connected, \(C^1\)--domain.
	Then, there exists a \emph{lifting operator} \(l_0:  W^{1-\frac{1}{q},q}(\Gamma)\to W^{1,q}(\Omega)\) such that
	\[ \begin{split}
			\|l_0v\|_{W^{1,q}(\Omega)} \leq C \|v\|_{W^{1-\frac{1}{q},q}(\Gamma)}
		\end{split}\]
	for every \(v \in C^\infty(\Gamma)\).
	The constant \(C\) depends only on \(\Omega\) and \(q\).
\end{lemma}

\begin{proof}
See  \cite{DiBenedetto1}, Chapter 10, Theorem 17.1c and Theorem 18.2c.
 \end{proof}

Equipped with the preceding technical tools, we now extend the identity \eqref{eq: Trace on Smooth} to an arbitrary \(\mathbf{u} \in W^{p}( \Div; \Omega)\).

\begin{thm}\label{thm:AppendixNormalTrace}
	Let \(\Omega\) be an open, bounded, \(C^1\)--domain. Then, there is a bounded linear operator
	\[ \begin{split}
			\tr_1: W^{p}( \Div; \Omega) \to W^{-1+\frac{1}{q} , \, p}(\Gamma)
		\end{split}\]
	satisfying the following properties.
	\begin{enumerate}[label = (\roman*)]
		\item The operator \(\tr_1\) is consistent with point-wise evaluation on the boundary for smooth function; that is,
        \[
            \tr_1(\mathbf{u}) = \mathbf{u}|_{\Gamma}\cdot \nu \quad \text{ for all }\, \mathbf{u} \in \left(\mathfrak{D} (\widebar{\Omega})\right)^d.
        \]
		\item  For each \(\mathbf{u} \in  W^{p}( \Div; \Omega)\),  we have that  \(\tr_1(\mathbf{u}) \in W^{-1+\frac{1}{q}, p}(\Gamma)\). That is, \( \tr_1(\mathbf{u})  \) is linear functional on \(W^{1-\frac{1}{q}, q}(\Gamma)\) bounded by

        \begin{equation} \label{eq:boundingTheLinearOperator1}
			      \| \tr_1(\mathbf{u})\|_{W^{-1+\frac{1}{q}, p}(\Gamma)} \leq C\|\mathbf{u}\|_{W^{p}( \Div; \Omega)}
		\end{equation}
              for some constant \( C \) independent of \( \mathbf{u} \).
	\end{enumerate}

\end{thm}

\begin{proof}
\begin{enumerate}
    \item For any smooth vector-field \( \mathbf{u} \in \left(\mathfrak{D} (\widebar{\Omega})\right)^d\) (see Definition~\ref{Def:TestFuns}), we define the Neumann trace of \(\mathbf{u}\) in the usual sense:
    \[
    \tr_1(\mathbf{u}) = \mathbf{u}|_{\Gamma} \cdot \nu.
    \]

    \item Fix \( v \in W^{1-\frac{1}{q},q}(\Gamma) \), and define the linear functional
    \[
        \Lambda_{v} \mathbf{u} = \langle \tr_1(\mathbf{u}), v \rangle
    \]
    for all \( \mathbf{u} \in \left(\mathfrak{D} (\widebar{\Omega})\right)^d \).
    Now, using the lifting operator from Lemma~\ref{lemma:liftingOp}, we obtain

    \begin{equation}\label{appenEq:1proof}
        \begin{split}
            |\Lambda_{v} \mathbf{u}| = |\langle \tr_1(\mathbf{u}), v \rangle |
            &= \big | \int_{\Gamma} \tr_1(\mathbf{u}) \, v \d s \big |\\[.25em]
            &=  \big |\int_\Omega \mathbf{u} \cdot D(l_0v)  + \Div (\mathbf{u}) \enskip l_0v \dx \big | \\[.25em]
            &\leq \|\mathbf{u}\|_{W^{p}( \Div; \Omega)} \enskip \| l_0v\|_{W^{1,q}(\Omega)}\\[.25em]
            &\leq C\|\mathbf{u}\|_{W^{p}( \Div; \Omega)} \enskip \|v\|_{W^{1-\frac{1}{q},q}(\Gamma)}.
		\end{split}
	\end{equation}
    Therefore, \( \Lambda_{v} \) is a bounded linear functional for all \( \mathbf{u} \in \left(\mathfrak{D} (\widebar{\Omega})\right)^d \cap W^{p}( \Div; \Omega)\).

    \item Because \(\left(\mathfrak{D} (\widebar{\Omega})\right)^d\) is dense in \(W^{p}( \Div; \Omega)\), we can extend \( \Lambda_{v} \) to act on arbitrary \(\mathbf{u} \in W^{p}( \Div; \Omega)\) as the limit
            \[
			     \Lambda_{v}\mathbf{u} = \lim\limits_{n\to \infty}  \Lambda_{v}\mathbf{u}_n
		    \]
		    for some sequence \(\{\mathbf{u}_n\}_{n=1}^{\infty} \subset \left(\mathfrak{D} (\widebar{\Omega})\right)^d \) such that \( \| \mathbf{u}-\mathbf{u}_n \|_{W^{p}( \Div; \Omega)} \to 0\) as \( n \to \infty \). This limit also preserves the bound
            \[
                |\Lambda_{v} \mathbf{u}|  \leq C\|\mathbf{u}\|_{W^{p}( \Div; \Omega)} \enskip \|v\|_{W^{1-\frac{1}{q},q}(\Gamma)}.
            \]

            \item  Therefore, we simply define \(\tr_1(\mathbf{u})\) for an arbitrary \(\mathbf{u} \in W^{p}( \Div; \Omega)\) according to the pairing
		      \[
			      \langle \tr_1(\mathbf{u}), v\rangle = \Lambda_{v} \mathbf{u}
		      \]
              for all \(v \in W^{1-\frac{1}{q}, q}(\Gamma)\).

            \item We finally remark that this pairing is, indeed, linear in \( v \) since
            \[
            \begin{split}
                \langle \tr_1(\mathbf{u}), a_1v_1+a_2v_2\rangle &= \lim\limits_{n\to \infty} \langle \tr_1(\mathbf{u}_n), a_1v_1+a_2v_2\rangle \\
                &= \lim\limits_{n\to \infty} (a_1 \, \langle \tr_1(\mathbf{u}_n), v_1\rangle + a_2\, \langle \tr_1(\mathbf{u}_n), v_2\rangle) \\
                &= a_1 \, \langle \tr_1(\mathbf{u}), v_1\rangle + a_2 \, \langle \tr_1(\mathbf{u}), v_2\rangle
            \end{split}
            \]
		    for some sequence \(\{\mathbf{u}_n\}_{n=1}^{\infty} \subset \left(\mathfrak{D} (\widebar{\Omega})\right)^d \) such that \( \| \mathbf{u}-\mathbf{u}_n \|_{W^{p}( \Div; \Omega)} \to 0\) as \( n \to \infty \). Additionally,
            \[
			  \|\tr_1(\mathbf{u}) \|_{  W^{-1+\frac{1}{q}, p}(\Gamma)}  = \sup_{\| v \| \leq 1}    \langle \tr_1(\mathbf{u}), v\rangle \leq C\|\mathbf{u}\|_{W^{p}( \Div; \Omega)}
            \]
            for some constant \( C >0 \) independent of \( \mathbf{u} \).
\end{enumerate}
    So, indeed, \(\tr_1(\mathbf{u}) \in W^{-1+\frac{1}{q}, p}(\Gamma)\) and this concludes the proof.
\end{proof}

\section*{Declarations}

\paragraph{\bf Ethical Statements:} Not applicable.

\paragraph{\bf Consent to participate:} Not applicable.

\paragraph{\bf Consent to publish:} Not applicable.

\paragraph{\bf Funding: } The research reported in this paper was funded through King Abdullah University of Science and Technology (KAUST) baseline funds and KAUST OSR-CRG2021-4674. A. M. Alharbi was additionally supported by Islamic University of Madinah, AD no.133212662/102931.

\bibliographystyle{abbrv}
\bibliography{mfgv7_nn.bib}

\end{document}